\documentclass[11pt,a4paper,oneside,headings=normal, parskip=half]{scrartcl}
\usepackage[left=3cm,right=3cm,top=2cm,bottom=2.5cm]{geometry}			
\usepackage[table,xcdraw]{xcolor}
\usepackage{setspace}	
\usepackage{aligned-overset}
\usepackage[bottom,hang]{footmisc}										
\usepackage[T1]{fontenc}												
\usepackage[utf8]{inputenc}											
\usepackage[ngerman, english]{babel}									
\usepackage{mathrsfs}
\usepackage{amsmath}													
\allowdisplaybreaks
\usepackage{amsthm}														
\usepackage{amstext}													
\usepackage{amssymb}													
\usepackage{amsfonts}													
\usepackage{dsfont} 	
\usepackage{csquotes}
\usepackage{mathtools}
\usepackage{url}														
\usepackage{graphicx}													
\usepackage{tocloft}													
\usepackage[bf,hang,nooneline,justification=centering]{caption}			
  
\usepackage{tabularx}													
\usepackage{multirow}													
\usepackage{booktabs}													
\usepackage{acronym}
\usepackage[titletoc,title]{appendix}									
\usepackage{biblatex}
\usepackage{placeins}
\usepackage{lmodern} 													
\usepackage{xcolor}														
\usepackage[hidelinks]{hyperref}													
\usepackage{stmaryrd}
\usepackage{comment}

\def\XXint#1#2#3{{\setbox0=\hbox{$#1{#2#3}{\int}$ }
\vcenter{\hbox{$#2#3$ }}\kern-.6\wd0}}

\usepackage{paralist}
\usepackage{enumitem}

\newcounter{assumption}
\newenvironment{assumption}[1][]{\refstepcounter{assumption}\par\medskip
   \noindent\textbf{Assumption~(A\theassumption). #1} \rmfamily}{\medskip}

\theoremstyle{definition} 
\newtheorem{theorem}{Theorem}[section]                      
\newtheorem{corollary}[theorem]{Corollary} 

\newtheorem{lemma}[theorem]{Lemma}
\newtheorem{definition}[theorem]{Definition}

\newtheorem{notation}[theorem]{Notation}

\newenvironment{beweis}{\begin{proof}[Proof]}{\end{proof}}

\newtheorem{beispiel}[theorem]{Example}
\newtheorem{bemerkung}[theorem]{Remark}

\makeatletter
\newcommand{\AlignFootnote}[1]{%
  \ifmeasuring@
  \else
    \iffirstchoice@
      \footnote{#1}%
    \fi
  \fi}
\makeatother

\DeclareMathOperator*{\fast}{-a.s.}
\DeclareMathOperator{\tr}{Tr}

\newcommand{\normlp}[1]{\left( \sum^{\infty}_{k=0} \norm{#1}_{HS}^2 \right)^{\frac{1}{2}}}
\DeclareMathOperator{\supp}{supp}
\newcommand{\dd}{\mathrm{d}}

\newcommand{\E}{\mathbb{E}}

\newcommand{\fastsicher}{\quad \W \fast}

\newcommand{\W}{\mathbb{P}}

\newcommand{\F}[2]{\mathcal{F}^{#1}_{#2}}

\newcommand{\skalar}[1]{\big( #1 \big)}
\newcommand{\skalarq}[1]{ \langle #1 \rangle }

\DeclareMathOperator{\ccc}{C}

\newcommand{\N}{\mathbb{N}}
\newcommand{\B}{\mathcal{B}}
\newcommand{\R}{\mathbb{R}}

\newcommand{\Rd}{\R^{d}}
\newcommand{\M}{\mathfrak{M}}

\newcommand{\sumkinfty}{\sum_{k=0}^\infty}

\renewcommand{\phi}{\varphi}
\renewcommand{\epsilon}{\varepsilon}
\newcommand{\eps}{\varepsilon}

\newcommand*{\fancym}{\mathfrak{M}}

\DeclareMathOperator{\llll}{L}
\DeclareMathOperator{\dive}{div}
\DeclareMathOperator{\dett}{det}

\newcommand{\abs}[1]{\left\vert #1 \right\vert}
\newcommand{\norm}[1]{\vert\vert {#1}\vert\vert}

\newcommand{\lp}[1]{\llll^{#1}}

\DeclareMathOperator{\TR}{tr}

\graphicspath{ {Bilder/} }
\usepackage{tikz}
\usetikzlibrary{arrows,automata,shapes,calc}
\usepackage{setspace}
\begin{document}

\begin{center}
    \Large
    \textbf{Intermittency Phenomena for Mass Distributions of Stochastic Flows with interaction}
        
        
    \vspace{0.4cm}
    Andrey Dorogovtsev\footnote{NAS Ukraine, andrey.dorogovtsev@gmail.com}, Alexander Weiß\footnote{University of Leipzig, alexander.weiss@math.uni-leipzig.de}
       
    \vspace{0.9cm}
\end{center}
\section{Introduction}
This article is devoted to the investigation of measure valued solution to stochastic differential equations with interactions. Such equations were introduced by \cite{Dorogovtsev2007} in order to describe motion with interaction. These equations are of the form: 

\begin{align}\label{Einl}
    \begin{cases}
dx(u,t) &= a(x(u,t),\mu_t)\dd t +\sum_{k=0}^\infty b_k(x(u,t),\mu_t)\dd B_k(t)
\\ \forall u\in \Rd \quad x(u,0)&= u \\
\mu_t&= \mu_0\circ{x^{-1}(\cdot,t)},
    \end{cases}
\end{align}
where $\mu_0$ describes the initial distribution of particles, $\mu_t$ is the pushforward measure of the system $x(\cdot,t)$. The coefficients are measurable functions $a:\Rd\times \M(\Rd) \to \Rd$ and $b_k:\Rd \times \M(\Rd)\to \R^{d\times d}$, where $\M(\Rd)$ denotes the space of probability measures on $\Rd$.

  For initial measures with Lebesgue densities we can investigate the behaviour of Lebesgue densities for $\mu_t$, as $t\to \infty$. 
  In this article we will investigate the occurrence of the intermittency phenomenon for such densities. 
Intermittency is the occurrence of of rare but high peaks, which shape the behaviour of the moments of the underlying quantity as $t\to \infty$, namely a significant growth increase for higher moments over smaller moments.
Such phenomena have been studied for a in the context of hydrodynamic turbulence (\cite{monin1975statistical},\cite{bruno2001identifying} and many more). Moreover  intermittency has been considered by \cite{zel1987intermittency} in random media, using moments to characterise the intermittency phenomenon. This has also been considered in connection with the parabolic Anderson model (\cite{carmona1995stationary},\cite{khoshnevisan2014analysis},\cite{gartner1990parabolic},\cite{chen2015precise} and many more) using the technique of moment Lyapunov exponents to determine intermittency in the parabolic Anderson model.

In this article, we will prove the existence of Lyapunov exponents for stochastic differential equations with interaction and prove an analogue to the Liouville formula for such equations. Finally we will prove, as the main result, existence of intermittency for the mass distribution, for initial measures with compact support. 
 Correspondingly this article is build as follows, in the second chapter we discuss the notion of stochastic differential equation with interactions and emphasise on the diffeomorphism property for the system \eqref{Einl}. The third chapter is devoted to the existence of Lebesgue densities for $\mu_t$, when $\mu_0$ has a Lebesgue density and discuss properties thereof. In the last chapter we define intermittency and show that under dissipativity conditions imposed on the coefficients, we achieve intermittency whenever the initial measure $\mu_0$ has compact support. 

\section{Stochastic Differential Equations with Interaction}
The main object of investigation are solutions to stochastic differential equations with interaction introduced in \cite{Dorogovtsev2007}. Before we start, we will introduce some notation involving the space of probability measures on $\Rd$ where from now on $d\ge 1$ shall be fixed. Let $\mathfrak{M}(X)$ denote the space of all probability measures on $X$ where $X$ is a topological space equipped with the corresponding Borel $\sigma$-algebra $\mathcal{B}(X)$. For simplicity we will denote $\fancym(\Rd)$ by $\fancym$ and $\mathcal{B}(\Rd)$ by $\mathcal{B}$. We will now introduce the notation for the metrisability of $\fancym$. 
\begin{definition}
Let $\mu,\nu \in \fancym $ then 
\begin{align*}
    C(\mu,\nu) := \{ \kappa \in \fancym(\Rd \times \Rd) \vert \kappa(\cdot\times \Rd)= \mu(\cdot), \kappa(\Rd \times \cdot)= \nu(\cdot)\}  
\end{align*}
is called space of couplings with marginals $\mu$ and $\nu$. 
\end{definition}
We can now define the Wasserstein space 
\begin{definition}
Define 
\begin{align*}
    \gamma: \fancym \times \fancym &\to \R \\
    (\mu,\nu) &\mapsto \gamma(\mu,\nu):= \inf_{\kappa \in C(\mu,\nu)} \underset{\Rd \times \Rd}{\int\int} \frac{\abs{u-v}}{1+\abs{u-v}} \kappa(\dd u, \dd v)
\end{align*}
$\gamma$ is called Wasserstein distance. 
\end{definition}

\begin{definition}
The space $(\fancym, \gamma)$ is called Wasserstein space 
\end{definition}
\begin{bemerkung}
It is well known that $(\fancym,\gamma)$ is a Polish space (e.g. \cite{villani2009optimal}).
\end{bemerkung}

\begin{definition}[\cite{Dorogovtsev2007}]
Let $(\Omega,\F{}{},\W)$ be a complete probability space and let $(B_k)_{k\ge 0}$ be $d$-dimensional Brownian motions. Let furthermore $a: \Rd \times \fancym \times \R_+ \to \Rd $ and $b_k: \Rd \times \fancym\times \R_+ \to \R^{d\times d}$ be measurable functions and $\mu_0 \in \fancym$, then the following stochastic differential equation 
\begin{align}\label{DefSDEint}
    \begin{cases}
    d x(u,t)&= a(x(u,t),\mu_t,t) \dd t + \sum_{k=0}^\infty b_k(x(u,t),\mu_t,t) \dd B_k(t)\\
\forall u\in \Rd \quad   x(u,0)&= u \\
\mu_t &= \mu_0\circ{x^{-1}(\cdot,t)} 
     \end{cases}
\end{align}
is called stochastic differential equation with interaction.  
\end{definition}

The key difference and difficulty compared to the usual  stochastic differential equations here, lies in the fact that not only does the value of the solution depend on the evolution in time but also on the evolution of the distribution in the space. \newline
In order to discuss well-posedness of such quantities we will first define strong solutions to  stochastic differential equations with interaction.
\begin{definition}[\cite{Dorogovtsev2007}] \label{wellposedness}
Consider the SDE situation in \eqref{DefSDEint} a map 

\begin{align*}
x&:\Rd \times \R_+\times \Omega \to \Rd\\
&(u,t,\omega) \mapsto x(u,t)(\omega)
\end{align*}
is called a unique strong solution to the stochastic differential equation with interaction \eqref{DefSDEint} if 
\begin{enumerate}[label=\roman*)]
    \item $x$ is measurable with respect to $\B\times \B([0,t]) \times \F{}{t}$ for all $t\ge 0$ on each interval $[0,t]$.
    \item $x$ suffices \eqref{DefSDEint} in the integral form. 
    \item Let $\Tilde{x}$ suffice $i)$ and $ii)$, then 
    \begin{align*}
        \W( \forall u\in \Rd,t\ge 0,x(u,t)=\Tilde{x}(u,t)) = 1
    \end{align*}
    holds.
\end{enumerate}
\end{definition}
 It turns out, that under similar assumptions as in the case of SDEs without interaction, we can obtain a well posedness result. 
\begin{theorem}\cite{Dorogovtsev2007}\label{ThExiSDE}
Consider the problem \eqref{DefSDEint}, there exists a solution in the notion of Definition \ref{wellposedness}, if  $a$ and $(b_k)_{k\ge 0}$ are jointly continuous and  
    there exists a constant $C>0$, such that for all $u,v\in \Rd$,$\mu,\nu\in \fancym$ and all $t\in [0,T]$ where $T>0$ is arbitrary, such that
    
    \begin{align*}
        \abs{a(u,\mu,t)-a(v,\nu,t)}+ \left(\sum_{k=0}^\infty \norm{b_k(u,\mu,t)-b_k(v,\nu,t)}_{HS}^2\right)^{\frac{1}{2}}\le C(\abs{u-v}+\gamma(\mu,\nu))
     \end{align*}
     where $\norm{\cdot}_{HS}$, denotes the Hilbert-Schmidt norm and
     \begin{align*}
         \abs{a(u,\mu,t)} + \left( \sum_{k=0}^\infty \norm{b_k(u,\mu,t)}^2_{HS}\right)^{\frac{1}{2}}\le C(1+\abs{u}) 
     \end{align*}
     
\end{theorem}

It is clear that, in case of existence of solutions, x is a solution to a SDE with random coefficients $\Tilde{a}(u,t)=a(u,\mu_t,t)$ and $\Tilde{b}_k(u,t,p)=b_k(u,\mu_t,t,p)$ , for all $t\ge 0$ and $k\ge 0$. Here $(\mu_t)_{t\ge0 }$ is the fixed measure valued process which is induced from the solution to the SDE with interaction. These random coefficients are almost surely Lipschitz for all $t\ge 0$ with respect to $u\in \Rd$. Therefore we can conclude similar properties as in the case of usual SDEs
\begin{corollary}
Consider the situation of Theorem \ref{ThExiSDE}, then the following relations hold.

    \begin{enumerate}[label=$\roman*)$]
        \item There exists a constant $C_T>0$ such that for all $p\ge 1$
        \begin{align*}
       \forall u\in \Rd \quad     \E(\sup_{0\le t\le T} \abs{x(u,t)}^p) \le C(1+\abs{u}^p)
        \end{align*}
        holds.
        \item There exists a constant $C_T>0$ such that for all $p\ge 1$
        \begin{align*}
        \forall u,v\in \Rd\quad    \E(\sup_{0\le t\le T} \abs{x(u,t)-x(v,t)}^p)\le C(\abs{u-v}^p)
        \end{align*}
    \end{enumerate}
\end{corollary}
\begin{beweis}
The proof of these assertions is a simple consequence of Grönwall's Lemma. 
"$i)$" Let $t\le T$. The Burkholder-Davis-Gundy inequality yields
\begin{align*}
&E(\sup_{0\le s\le t} \abs{x(u,s)}^p)
\\\le& Ct \left(\abs{u}^p+\E(\int_0^t \abs{a(x(u,s),\mu_s,s)}^p \dd s) + \E(\int_0^t \left(\sum_{k=0}^\infty \abs{b_k(x(u,s,\mu_s,s))}^2)\right)^{\frac{p}{2}} \dd s \right)  
\\\le& Ct\Big(\abs{u}^p + \E\left(\int_0^t (1+\abs{x(u,s)}^p) \dd s\right)\Big).
\end{align*}
 The claim is obtained by applying the Grönwall Lemma to 
\begin{align*}
    g(t)= \E(\sup_{0\le s\le t} \abs{x(u,s)}^p).
\end{align*}
We thus get 
\begin{align*}
    \E(\sup_{0\le t\le T }\abs{x(u,t)}^p) \le C(1+\abs{u}^p)
\end{align*}
for all $T\ge 0$ where C depends on $T$. The result of $"ii)"$ is obtained in exactly the same way. 
\end{beweis}
\begin{beispiel}
Consider a Lipschitz bounded function $f_l:\Rd \times \R^{dl} \to \Rd $ and $l\ge 1$, then 
\begin{align*}
    a(u,\mu):= \sum^n_{k=1}\int_{\Rd}\dots \int_{\Rd} f_k(u,v_1,\dots,v_k) \mu(\dd v_1)\dots\mu(\dd v_k) 
\end{align*}
suffices the Lipschitz property in Theorem \ref{ThExiSDE}.
\end{beispiel}
\begin{beweis}
In order to reduce notation we will only prove the case in which the functions $f_l$ for all $l\in \{0,\dots, n\}$
are independent of $s$. The proof can be carried out in exactly the same way as below. We begin by estimating the following.
\begin{align*} 
   \Delta&:= \sum_{k=1}^n \int_{\Rd}\dots \int_{\Rd} f_k(u_1,v_1,\dots v_k) \mu(\dd v_1),\dots,\mu(\dd v_k) \\&- \int_{\Rd}\dots \int_{\Rd} f_k(u_2,v_1,\dots v_k) \nu(\dd v_1),\dots,\nu(\dd v_k)
\end{align*}
for all $u_1,u_2 \in \Rd, \mu,\nu \in \mathfrak{M}$ Let $\kappa$ be a coupling on $\Rd\times\Rd$  such that its marginal distributions coincide with $\mu$ and $\nu$. Denote $\phi(r)= \frac{r}{1+r}$ for all $r\ge 0$. Then 
\begin{align}\label{Delta Absch}\begin{split}
    \abs{\Delta} &\le\sum_{k=1}^n  \underset{\Rd \times \Rd}{\int\int} \dots  \underset{\Rd \times \Rd}{\int\int}\big\vert f_k(u_1;v_1,\dots,v_k)-f_k(u_1,v^\prime_1,\dots,v^\prime_k)\big\vert \\&+ \abs{f_k(u_1;v^\prime_1,\dots,v^\prime_k)-f_k(u_2,v^\prime_1,\dots,v^\prime_k)} \kappa(\dd v_1, v^\prime_1),\dots, \kappa(\dd v_k, v^\prime_k) \\&\le\sum_{k=1}^n  \underset{\Rd \times \Rd}{\int\int} \dots  \underset{\Rd \times \Rd}{\int\int}\abs{f_k(u_1;v_1,\dots,v_k)-f_k(u_1,v^\prime_1,\dots,v^\prime_k)}\\&\kappa(\dd v_1, v^\prime_1),\dots, \kappa(\dd v_k, v^\prime_k)+C\abs{u_1-u_2}\\\overset{*}&{\le} \sum_{k=1}^n \Tilde{C} \underset{\Rd \times \Rd}{\int\int} \dots  \underset{\Rd \times \Rd}{\int\int}\phi(\sum_{l=1}^k\abs{v_l-v\prime_l})\kappa(\dd v_1, v^\prime_1),\dots, \kappa(\dd v_k, v^\prime_k)\\& +C\abs{u_1-u_2}\\\overset{}&{\le}\sum_{k=1}^n \Tilde{C}\sum_{l=1}^k \underset{\Rd \times \Rd}{\int\int} \dots  \underset{\Rd \times \Rd}{\int\int}\phi(\abs{v_l-v^
    \prime_l})\kappa(\dd v_1, v^\prime_1),\dots, \kappa(\dd v_k, v^\prime_k)\\& +C\abs{u_1-u_2}.\end{split}\end{align}Here $*$ can be  obtained as follows: consider $(v_1,\dots,v_k),(v^\prime_1,\dots,v^\prime_k) \in \R^k$ for all $k\le n$ such that $\sum_{l=1}^k\abs{v_1-v^\prime_1}>\eps$ for $\eps>0$ fixed. Boundedness of $f$ and monotonicity of $\phi$ imply: 
    \begin{align*}
        \abs{f_l(u_1;v_1,\dots,v_k)-f_l(u_1,v^\prime_1,\dots,v^\prime_k)}\le \underbrace{\frac{2\sup_{l=1,\dots,n ;x\in \R^{k+1}} \abs{f_l(x)}}{\phi(\eps)}}_{:=C_1} \phi(\eps)\le C_1 \phi(\sum_{l=1}^k\abs{v_l-v^\prime_l}).
    \end{align*}
    On the other hand we can conclude for all $(v_1,\dots,v_k),(v^\prime_1,\dots,v^\prime_k) \in \R^{dk}$ such that $\sum_{l=1}^k\abs{v_l-v^\prime_l}\le \eps$
    \begin{align*}
        \abs{f(u_1;v_1,\dots,v_k)-f(u_1,v^\prime_1,\dots,v^\prime_k)}\le C \sum_{l=1}^k \abs{v_l-v_l^\prime} \le \underbrace{(1+\eps)C}_{:=C_2} \phi(\sum_{l=1}^k \abs{v_l-v_l^\prime}) 
    \end{align*}
    by Lipschitz continuity of $f$.
    Hence we get
    \begin{align*}
         \abs{f(u_1;v_1,\dots,v_k)-f(u_1,v^\prime_1,\dots,v^\prime_k)}\le \max\{C_1,C_2\}\phi(\sum_{l=1}^k \abs{v_l-v_l^\prime}).
    \end{align*}
    Thus \eqref{Delta Absch} finally yields 
    \begin{align*}
        \abs{\Delta} \le K(\abs{u_1-u_2} + \gamma(\mu,\nu)) 
    \end{align*}
    for some constant $K>0$, since $\kappa$ was an arbitrary coupling.

\end{beweis}
Since the notion of stochastic differential equations with interaction is established, we want to go further and investigate the properties of the measure valued process $(\mu_t)_{t\ge 0}$. Her we consider the following question. If the initial data $\mu_0$ is absolutely continuous with respect to the Lebesgue measure on $\Rd$, is $\mu_t$ also absolutely continuous with respect to the Lebesgue measure? In order to give an answer to this question, let us first observe the following relation. Let $x$ be a solution to a SDE with interaction and assume that $x(\cdot,t)$ is as smooth as needed for all $t\ge 0$ and let $p_0$ be the Lebesgue density of $\mu_0$. Then we have for a bounded measurable function $\phi:\Rd\to \Rd$ the following:

\begin{align}
\begin{split}\label{Ideediffeo}
\int_{\Rd} \phi(u) \mu_t(\dd u) &= \int_{\Rd} \phi(x(u,t)) \mu_0(\dd u)   \\
&= \int_{\Rd} \phi(x(u,t)) p_0(u) \dd u = \int_{\Rd} \phi(u) p_0(x^{-1}(u,t))\dett(Dx^{-1}(u,t)) \dd u 
\end{split}
\end{align}
Hence we get, that if $x(\cdot,t)$ is diffeomorphic, then $\mu_t$ is absolutely continuous with respect to the Lebesgue measure for all $t\ge 0$ almost surely. More over the density, which we will denote by $(p_t)_{t\ge 0}$ is of the explicit form 
\begin{align*}
    p_t(u) = p_0(x^{-1}(u,t)) \abs{\det(Dx^{-1}(u,t))}
\end{align*}
\subsection{Stochastic Flow of Diffeomorphisms with Interaction}
We will show in this section, that the desired diffeomorphism property holds, for solutions to the problem \eqref{DefSDEint}. We will utilise the tools given by the theory of so called stochastic flows, established by Hiroshi Kunita in \cite{kunita1997}.
\begin{definition}[\cite{kunita1997},p.114]
a family of random maps $(\phi_{s,t})_{s\le t\in [0,T]}$ for $T\ge 0$ such that $\phi_{s,t}:\Rd \times \Omega \to \Rd$ is called a forward stochastic flow of $\ccc^1$-diffeomorphism, if there exists a $\W$-null set $N$, such that for all $\omega \in N^c$, we have
\begin{enumerate}[label=$\roman*)$]
    \item $\phi_{s,u}(\omega)\equiv \phi_{t,u}(\omega) \circ \phi_{s,t}(\omega)$ for all $s\le u, t\le u, s\le t \in [0,T]$. 
    \item For all $u\in \Rd$ the property $\phi_{s,s}(u,\omega) =  u$ holds for all $s\ge 0$. 
    \item The map $\phi_{s,t}(\omega):\Rd \to \Rd$ is a $\ccc^1$-diffeomorphism for all $s\le t\in [0,T]$. 
    
\end{enumerate}
\end{definition}

We want to apply Theorem 4.6.5 in \cite{kunita1997}, which yields the existence of a version, to the solution $x$, such that this version is a flow of diffeomorphisms. In order to do this we have to transform our the problem \eqref{DefSDEint} into the language of Kunita (\cite{kunita1997}). 
\begin{assumption}\raisebox{\ht\strutbox}{\hypertarget{(A1)}{}} 
The coefficients $a$ and $b_k$ in \eqref{DefSDEint} suffice the conditions of Theorem \ref{ThExiSDE}.
\end{assumption}
Consider the problem \eqref{DefSDEint}, we can now fix the measure valued process $(\mu_t)_{t\ge 0}$ induced from the solution to the SDE with interaction. As already emphasised, the solution to the SDE with interaction can be written into a form of well known SDEs with random coefficients in the Kunita sense. Namely consider the semimartingale 

\begin{align*}
    F(u,t):= \int_0^t a(u,\mu_t,t) \dd t +\sum_{k=0}^\infty \int_0^t b_k(u,\mu_t,t) \dd B_k. 
\end{align*}
We shall fix this semiamrtingale $F$. From the definition, we can immediately tell that a solution $x$ to \eqref{DefSDEint} is also a solution to a SDE in the sense of Kunita \cite{kunita1997} namely
\begin{align*}
x(u,t)= u +\int_0^t F(x(u,t),\dd t).
\end{align*}
We will finally apply Theorem 4.5.6 in \cite{kunita1997}, in order to get a modification of solutions, which are  diffeomorphisms. It is natural to assume, that the coefficients need to be differentiable, in general we need to the following assumptions. 
\begin{assumption}\raisebox{\ht\strutbox}{\hypertarget{(A2)}{}} 
The coefficients $a$ and $(b_k)_{k\ge 0}$ fulfil the following properties for $C>0$
\begin{enumerate}[label=$\roman*)$]
    \item $a$ and $(b_k)_{k\ge0}$ are continuously differentiable, with respect to $u$ for all $\mu\in \fancym,t\ge0$.
     \item there exists $0<\delta\le 1$ such that for all $u,v\in \Rd,\mu\in \fancym$ and $t\ge0$
     \begin{align*}
        \abs{Da(u,\mu,t)-Da(v,\mu,t)}+ \left(\sum^\infty_{k=0} \norm{Db_k(u,\mu,t)-Db_k(v,\mu,t)}_{HS}^2\right)^{\frac{1}{2}} \le C \abs{u-v}^\delta
      \end{align*}
      where $\norm{\cdot}_{HS}$ denotes the Hilbert Schmidt norm in $\R^{d\times d\times d}$.
      \item For all  $t\ge 0$ we have 
      \begin{align*}
    \forall \mu \in \fancym \quad      \sup_{u\in \Rd} \left(\sum^\infty_{k=0} \norm{Db_k(u,\mu,t)}_{HS}^2\right)^{\frac{1}{2}}\le C
      \end{align*}
\end{enumerate}
\end{assumption}
with these assumptions at hand we obtain a version of diffeomorphisms.
\begin{theorem}\label{Thdiffeo}
Let the coefficients $a$ and $(b_k)_{k\ge 0}$ suffice the assumptions \hyperlink{(A1)}{(A1)} and \hyperlink{(A2)}{(A2)}, then there exists a version $\Tilde{x}$ still denoted by $x$ to the problem \eqref{DefSDEint}, such that $\phi_{s,t}=x_s(u,t)$ where 
\begin{align*}
    x_s(u,t)= u+ \int_s^t F(x_s(u,t),\dd t)
\end{align*}
 is a stochastic flow of diffeomorphisms

\end{theorem}
\begin{beweis}
We have to show that $(F(\cdot,t))_{t\ge0 }$ suffices the regularity conditions of Theorem 4.5.6 in \cite{kunita1997}. For this we will use the exact notation of \cite{kunita1997} in chapter $3$ and $4$. Consider the (semi-)norms for functions on the space of $m$-times continuously differentiable bounded functions $C^m(\Rd, \R^e)$ where $e\ge 1$. First note for multi indices with non-negative integers $\alpha=(\alpha_1,\cdot,\alpha_d)$, we shall use the notation $D^\alpha$ or $D^\alpha_u$ where 
\begin{align*}
    D^\alpha_x = \frac{\partial^{\abs{\alpha}}}{(\partial x_1)^{\alpha_1}\dots(\partial x_d)^{\alpha_d}},
\end{align*}
with $\abs{\alpha}= \sum_{i=1}^d \alpha_i$. We set 
\begin{align*}
    \norm{f}_m = \sup_{u\in K} \frac{\abs{f(u)}}{(1+\abs{u})} + \sum_{1\le \abs{\alpha}\le m} \sup_{u\in K} \abs{D^\alpha f(u)}
\end{align*}
for any $K\subset \Rd$.
Further define the space $\ccc^{m,\delta}(\Rd,\R^e)$ as the functions $f\in \ccc^m(\Rd,\R^e)$ such that $D^\alpha f$ is $\delta$-Hölder continuous for all $\abs{\alpha}=m$. Then define
\begin{align*}
    \norm{f}_{m+\delta,K} = \norm{f}_{m,K} +\sum_{\abs{\alpha}=m}\sup_{u,v\in K,u\neq v} \frac{\norm{D^\alpha f(u)-D^\alpha f(v)}_{HS}}{\abs{u-v}^\delta}.
\end{align*}
  Analogously we can define the same for functions $g:\Rd\times \Rd\to \R^{e}$ by 
  \begin{align*}
      \norm{g}^\prime_{m:K} = \sup_{u,v} \frac{\abs{g(u,v)}}{(1+\abs{u})(1+\abs{v})} + \sum_{1\le\abs{\alpha}\le m} \sup_{u,v\in K} \norm{D^\alpha_u D^\alpha_v g(u,v)}_{HS}
  \end{align*}
and 
\begin{align*}
    \norm{g}^\prime_{m+\delta:K} = \norm{g}^\prime_{m:k} +\sum_{\abs{\alpha}=m} \sup_{\underset{u_1\neq u_2, v_1\neq v_2}{u_1,u_2,v_1,v_2\in K} } \frac{\norm{D^\alpha_uD^\alpha_v\left( g(u_1,v_1)-g(u_2,v_1)-g(u_1,v_2)+g(u_2,v_2)\right)}_{HS}}{\abs{u_1-u_2}^\delta \abs{v_1-v_2}^\delta}. 
\end{align*}
Consider now the semimartingale $F(\cdot,t)$ it can be desomposed in $A(\cdot,t)= \int_0^t a(\cdot,\mu_s,s) \dd s$ and $M(\cdot,t)= \sum_{k=0}^\infty \int_0^t b_k(\cdot,\mu_s,s) \dd B_k(s)$. In order to apply Theorem 4.5.6. \cite{kunita1997} it suffices to show that there exists a constant $C>0$ such that almost surely  
\begin{enumerate}[label=$\roman*)$]
    \item For all $T>0$  
    \begin{align*}
        \norm{a(\cdot,\mu_t,t)}_{1+\delta,\Rd} \le C 
    \end{align*}
    for all $t\in [0,T]$.
    \item For all $T>0$ we have for $g(u,v,t)= \sum_{k=0}^\infty b_k(u,\mu_t,t) b^*_k(v,\mu_t,t)$ where $*$ denotes the adjoint. For all $u,v\in \Rd$ and $t\in [0,T]$. 
    \begin{align*}
        \norm{g(\cdot,\cdot,t)}_{m+\delta:\Rd}
    \end{align*}

\end{enumerate}
$i)$ is obvious, from the assumptions \hypertarget{(A1)}{(A1)}  and \hypertarget{(A2)}{(A2)}, $ii)$ follows immediately. Observe first, that  by Assumption \hypertarget{(A2)}{(A2)} we can interchange differentiation and summation. In order to do so we only need to check, that the sum of partial derivatives converges uniformly on some compact. Let $R>0$ and consider the closed ball with radius R $\overline{B}_R(0)$. Let $\eps>0$ be arbitrary, by compactness we can find $u_1,\dots,u_n\in \overline{B}_R(0)$ such that $\overline{B}_R(0)\subset \bigcup^n_{i=1} B_\eps(u_i) $. Now one can find $N\in \N$ for all $i=1,\dots,n$ and all $j,p,l=1,\dots,d$ such that 
\begin{align*}
    \sum_{k=N+1}^\infty \abs{\frac{\partial b_k^{j,p}}{\partial u_l}(u_i,\mu,t)}^2<\eps
\end{align*}
for all $t\ge 0,\mu\in \fancym$ by assumption \hypertarget{(A2)}{(A2)} $iii)$. Hence we get for all $u \in B_R(0)$, we can find $i=1,\dots,n$ such that $\abs{u-u_i}<\eps
$ 
\begin{align*}
    &\left(\sum^\infty_{k=N+1} \abs{\abs{\frac{\partial b^{j,p}}{\partial u_l}(u,\mu,t)}^2} \right)^{\frac{1}{2}}\\
    &\le \left(\sum^\infty_{k=N+1} \abs{\frac{\partial b^{j,p}}{\partial u_l}(u,\mu,t)-\frac{\partial b^{j,p}}{\partial u_l}(u_i,\mu,t)}^2 \right)^{\frac{1}{2}} +\left( \sum^\infty_{k=N+1} \abs{\frac{\partial b^{j,p}}{\partial u_l}(u_i,\mu,t)}^2 \right)^{\frac{1}{2}} \le 2\eps.
\end{align*}
Hence the series converges uniformly and thus we can interchange differentiation and summation in the following way:

\begin{align*}
    D^\alpha_uD^\alpha_v g(u,v,t)= \sum_{k=0}^\infty D^\alpha_ub_k(u,\mu_t,t) D^\alpha_v\left( b_k(v,\mu_t,t)\right)^*
\end{align*}
 where $*$ denotes the adjoint operator.  for all $\abs{\alpha}\le1$. Then we arrive at 
\begin{align*}
    \abs{g(u,v,t)}\le  \normlp{b_k(v,\mu_t,t)} \normlp{b_k(u,\mu_t,t)} \le C(1+\abs{u})(1+\abs{v}),
\end{align*}
furthermore
\begin{align*}
    D^\alpha_uD^\alpha_v(g(u,v,t))\le \normlp{D^\alpha b_k(v,\mu_t,t)}\normlp{D^\alpha b_k(u,\mu_t,t)}\le C
\end{align*}
for all $u,v\in\Rd$ and all $t\ge 0$.  Finally, we have
\begin{align*}
   & \norm{D^\alpha_uD^\alpha_v\left( g(u_1,v_1)-g(u_2,v_1)-g(u_1,v_2)+g(u_2,v_2)\right)}_{HS} \\
   &= \normlp{D^\alpha_u b_k(u_1,\mu_t)-D^\alpha_u b_k(u_2,\mu_t)} \normlp{D^\alpha_v b_k(v_1,\mu_t)-D^\alpha_u b_k(v_2,\mu_t)}\\
   &\le C (\abs{u_1-u_2}^\delta \abs{v_1-v_2}^\delta).
\end{align*}

Hence we get 
\begin{align*}
    \norm{\Tilde{a}(\cdot,t)}_{1+\delta,\Rd} +\norm{g(\cdot,\cdot,t)}^\prime_{1+\delta,\Rd} \le C
\end{align*}
for some constant $C>0$ and for all $t\le T$, where $T\ge0$.
This yields, by Theorem 4.5.6 in \cite{kunita1997}, the existence of a version of $x$ such that $x(\cdot,t)$ is a diffeomorphism almost surely for all $t\ge 0$.
\end{beweis}
Theorem 3.3.3 in \cite{kunita1997} furthermore implies that the derivative solves a linear SDE with interaction.
\begin{corollary}\label{DifferentialSDE}
The derivative $Dx(u,t)$ solves the following SDE with interaction 

\begin{align*}
    dDx(u,t)&= Da(x(u,t),\mu_t,t)Dx(u,t)\dd t +\sum_{k=0}^\infty Db_k(x(u,t),\mu_t,t)\dd B_k(t)Dx(u,t)\\
  \forall u \in \Rd \quad  Dx(u,0) &= I_{d\times d}
\end{align*}
where $I_{d\times d}$ denotes the identity matrix in $d$-dimensions.
 \end{corollary}

\section{Mass Distributions of Stochastic Flows with Interaction}
We will now turn our focus back to the main object of investigation. Namely the measure valued process $(\mu_t)_{t\ge 0}$ induced from the SDE with interaction \eqref{DefSDEint}. Now we have the tools at hand, such that we can show the existence of a Lebesgue density for $(\mu_t)_{t\ge 0}$ almost surely for all $t\ge 0$. 
\begin{theorem}
Consider the problem \eqref{DefSDEint} under the assumptions \hyperlink{(A1)}{(A1)} and \hyperlink{(A2)}{(A2)}. Assume that the initial measure $\mu_0$ is absolutely continuous with respect to the Lebesgue measure. Then also $(\mu_t)_{t\ge 0}$ is also absolutely continuous with respect to the Lebesgue measure for all $t\ge 0$ almost surely. 
\end{theorem}
 The theorem follows immediately from \eqref{Ideediffeo} combined with Theorem \ref{Thdiffeo}, if we can show that $\det(Dx^{-1}(u,t))$ is not zero Lebesgue almost everywhere. Hence we will show the latter by proving an analogue to the deterministic Liouville formula.
\begin{assumption}\raisebox{\ht\strutbox}{\hypertarget{(A3)}{}} 
We will from now on assume, if not mentioned otherwise, that the initial measure of the problem \eqref{DefSDEint} is absolutely continuous with respect to the Lebesgue measure. 
\end{assumption}

\begin{notation}
 We will from now on denote the Lebesgue density corresponding to $\mu_0$ and $\mu_t$ as $p_0$ and $p_t$ respectively for all $t\ge 0$.
\end{notation}
From \eqref{Ideediffeo} we actually get a concrete form for the Lebesgue densities $(p_t)_{t\ge 0}$ namely $p_t= p_0(x^{-1}(\cdot,t)) \abs{\dett(Dx^{-1}(u,t))}$. We can actually give an explicit form, of the occurring determinant. The determinant is a polynomial, and $\left(Dx(u,t)\right)_{u\in \Rd,t\ge0}$ is the solution to a linear equation due to Corollary \ref{DifferentialSDE}. We can thus apply Itô's formula in order to compute the concrete form. But before we do that, good representations of partial derivatives of the determinant are required. 
\begin{lemma}\label{fruitful}
Let $A\in \R^{d\times d}$ be invertible and $B\in \R^{d\times d}$ then we have
 \begin{enumerate}[label= $\roman*$)]
    \item 
Consider the determinant function as a function of the components of the matrix, i.e. the entries of the matrix.  Let $A= (A_{i,j})_{i,j=1,\dots,d }$ and $B=(B_{i,j})_{i,j=1,\dots,d}$.

\begin{align*}
    \frac{\partial \det(A)} {\partial A_{i,j}} = (-1)^{i+j} M_{i,j}(A)
\end{align*}
where $M_{i,j}(A)$ is the determinant of the submatrix of $A$ where the $i$-th row and the $j$-th column has been erased. 
\item 
\begin{align*}
    \sum^d_{i,j=1} \frac{\partial \det(A)}{\partial A_{i,j}} (BA)_{i,j} = \det(A)\TR(B)
\end{align*}
\item It holds that 
\begin{align*}
\sum_{i,j,k,l=1}^d \frac{\partial^2 \dett(A)}{\partial A_{i,j} \partial A_{k,l} } (BA)_{i,j}(BA)_{k,l} = (\tr(B)^2-\tr(B^2)) \dett(A)    
\end{align*}
\end{enumerate}

\end{lemma}
\begin{proof}
"$i)$" \newline
 We denote the determinant of the submatrix of a matrix $A\in \R^{d\times d}$, which results from removing the $i$-th row and the $k$-th column, with $M_{i,k}(A)$ for $i,k=1,\dots, d$.
 Then the Laplace expansion of the determinant yields, for $A\in \R^{d\times d}$ and  $\pi_{i,k}(A) = A_{i,k}$ for all $A\in \R^{d\times d}$, that
\begin{align*}
\frac{\partial \det(A)}{\partial A_{i,j}} &= \frac{\partial}{\partial A_{i,j}} \sum_{k=1}^d (-1)^{i+k} M_{i,k}(A) \pi_{i,k}(A)
    \\&= \sum_{k=1}^d (-1)^{i+k}\frac{\partial \pi_{i,k}} {\partial A_{i,j}} (A) M_{i,k}(A)+ A_{i,k}  \frac{\partial M_{i,k}}{\partial A_{i,j}} (A) = (-1)^{i+j} M_{i,j}(A)
\end{align*}
holds. Where the last equality follows from
\begin{align*}
    \frac{\partial M_{i,k}}{\partial A_{i,j}}(A) =0
\end{align*}
since $M_{i,k}$ does not depend on $A_{i,j}$ because the $i$-th row was removed. Furthermore it is obvious that 
\begin{align*}
    \frac{\partial \pi_{i,k}}{\partial A_{i,j}} (A) = \delta_{k,j}
\end{align*}
holds for all $i,j,k=1,\dots, d$ with Kronecker's delta $\delta_{k,j}$.
\newline
"$ii)$"
\newline
by $i)$ we know 
\begin{align}\label{glgdet}\begin{split}
     \sum_{i,j=1}^d \frac{\partial \det (A)}{\partial A_{i,j}} (BA)_{i,j} &= \sum_{i,j=1}^d (-1)^{i+j} M_{i,j}  \sum_{k=1}^d B_{i,k}A_{k,j}
        \\\overset{*}&{=} \sum_{i,j=1}^d (-1)^{i+j} M_{i,j} B_{i,i} A_{i,j}= \TR(B) \det(A)
\end{split}
\end{align}  
where $*$ follows easily by considering
\begin{align}\label{deter}
    \sum_{j=1}^d (-1)^{i+j} M_{i,j}(A) A_{k,j} = 0
\end{align}
for all $k\neq i$, since \eqref{deter} is the determinant of the matrix $A$ where the $i$-th row has been replaced by the vector $(A_{k,1},\dots,A_{k,d})$. This determinant is obviously $0$ since the matrix has at most rank $d-1$.\newline
"$iii)$" From $ii)$ we can derive the formula, by looking at 
\begin{align}\label{glg2teabl}
    \sum_{i,j,k,l=1}^d \frac{\partial}{\partial A_{k,l}} \left( \frac{\partial \dett(A)}{\partial A_{i,j}} (BA)_{i,j}(BA)_{k,l}\right) = \tr(B)\sum_{k,l=1}^d \frac{\partial }{\partial A_{k,l}} \left( \det(A) (BA_{k,l})  \right)  
\end{align}
Let us firs compute the left-hand side. Here we get 
\begin{align*}
     & \sum_{i,j,k,l=1}^d \frac{\partial}{\partial A_{k,l}} \left( \frac{\partial \dett(A)}{\partial A_{i,j}} (BA)_{i,j}(BA)_{k,l}\right) = \sum_{i,j,k,l,m,n=1}^d \frac{\partial}{\partial A_{k,l}} \left( \frac{\partial \dett(A)}{\partial A_{i,j}} B_{i,n}A_{n,j}B_{k,m}A_{m,l}\right)\\
     &+\sum_{i,j,k,l=1}^d  \frac{\partial^2 \dett(A)}{\partial A_{i,j} \partial A_{k,l} } (BA)_{i,j}(BA)_{k,l} + \underbrace{\sum_{i,j,k,l,n,m=1}^d \frac{\partial \dett(A)}{\partial A_{i,j}} B_{i,n}\frac{\partial \pi_{n,j}(A)}{\partial A_{k,l}}B_{k,m}A_{m,l}}_{:=C}\\
     &+ \underbrace{\sum_{i,j,k,l,n,m=1}^d \frac{\partial \dett(A)}{\partial A_{i,j}} B_{i,n} A_{n,j}B_{k,m}\frac{\partial \pi_{m,l}(A)}{\partial A_{k,l}}}_{:=D}
 \end{align*}
 for the second term we get 
 \begin{align*}
     C&= \sum_{i,k,l,m=1}^d \frac{\partial \dett(A)}{\partial A_{i,l}} B_{i,k}B_{k,m}A_{m,l}\overset{(*)}{=} \sum^d_{i,k=1}B_{i,k}B_{k,i} \sum^d_{l=1} \frac{\partial \dett(A)}{\partial A_{i,l}} A_{i,l} \overset{(**)}{=} \det(A)\sum_{i,k=1}^d B_{i,k}B_{k_i}\\
     &= \dett(A) \tr(B^2)
 \end{align*}
 where $(*)$ follows in exactly the same way as in \eqref{glgdet} and $(**)$ follows from $i)$ combined with the determinant decomposition formula. For the second term we get 
 \begin{align*}
     D &= \sum_{i,j,k,l,n=1}^d \frac{\partial \dett(A)}{\partial A_{i,j}} B_{i,n} A_{n,j}B_{k,k}\overset{(*)}{=} \sum_{i,j,k,l= 1}^d \frac{\partial \dett(A)}{\partial A_{i,j}} B_{i,i} A_{i,j}B_{k,k}\\
     & \sum_{l=1}^d\sum_{i,k=1}^d B_{i,i}B_{k,k} \sum^d_{j=1}A_{i,j}\frac{\partial \dett(A)}{\partial A_{i,j}} = d\tr(B)^2 \dett(A)
 \end{align*}
 Now that we have a representation for the left-hand side we want to rewrite the right-hand side of \eqref{glg2teabl}. Consider 
 \begin{align*}
      &\tr(B)\sum_{k,l=1}^d \frac{\partial }{\partial A_{k,l}} \left( \det(A) (BA_{k,l})  \right)\\
      =& \tr(B)\left(\sum_{k,l=1}^d \frac{\partial \det(A)}{\partial A_{k,l}} (BA_{k,l})  + \det(A) \sum_{n=1}^d B_{k,n}\frac{\partial\pi_{n,l}(A)}{\partial A_{k,l}}\right)
 \end{align*}
 the first sum is equal to $\tr(B)\det(A) $, by $ii)$ the second one however, equals to 
 \begin{align*}
   \sum_{k,l=1}^d  \det(A) B_{k,k} = d \tr(B) \det(A)
 \end{align*}
 In total we get from all the previous computation and \eqref{glg2teabl}, the following formula
 
 \begin{align*}
     &\sum_{i,j,k,l=1}^d   \frac{\partial^2 \dett(A)}{\partial A_{i,j}\partial A_{k,l}} (BA)_{i,j}(BA)_{k,l} \\&= \tr(B)^2 \det(A) + d \tr(B)^2 \det(A) -d\tr(B)^2 \det(A) -\tr(B^2)\det(A) \\
     &= (\tr(B)^2-\tr(B^2)) \det(A)
 \end{align*}
\end{proof}
With this Lemma at hand, we can prove the SDE analogue of Liouville's theorem. In the following we will denote by $\dive(g(\cdot))(v)$ the divergence of $g(\cdot)$ evaluated at $v\in \Rd$, the same holds for the differential $D$.
\begin{theorem}\label{Liouville}
Under the Assumptions \hyperlink{(A1)}{(A1)} and \hyperlink{(A2)}{(A2)} we have 

\begin{align*}
    \det(Dx(u,t))&= \exp\Big(\int_0^t \dive(a(\cdot,\mu_s,s))(x(u,s))-\frac{1}{2} \sum_{k=0}^\infty \sum_{p=1}^d\tr\left( \left(Db_k^{\cdot,p}(x(u,s),\mu_s,s)\right)^2 \right) \dd s \\
    &+ \sum_{k=0}^\infty \sum_{p=1}^d\int_0^t \dive(b_k^{\cdot,p}(\cdot,\mu_t,t))(x(u,t)) \dd B^p_k(t) \Big) 
\end{align*}
\end{theorem}
\begin{beweis}
The theorem follows from the representation of the derivative. By applying Itô's formula, we  get from Lemma \ref{fruitful} \begin{align*}
    \det(Dx(u,t))-1&= \sum^d_{i,j=1}\int_0^t \frac{\partial \det(Dx(u,s))}{\partial A_{i,j}}\dd D^{i,j}x(u,s) \\
    &+\frac{1}{2} \sum_{i,j,k,l=1}^d \int_0^t \frac{\partial^2 \det(Dx(u,s))}{\partial A_{i,j} \partial A_{k,l}} \dd \skalar{D^{i,j}x(u,\cdot),D^{k,l}x(u,\cdot)}_s\\
    &=  \sum^d_{i,j=1}\int_0^t \frac{\partial \det(Dx(u,s))}{\partial A_{i,j}} (Da(\cdot,\mu_s,s)(x(u,s))Dx(u,s))_{i,j} \dd s\\
    &+ \sum^\infty_{k=0}\sum^d_{i,j,p=1}\int_0^t \frac{\partial \det(Dx(u,s))}{\partial A_{i,j}}(Db_k^{\cdot,p}(\cdot,\mu_s,s)(x(u,s))Dx(u,s))_{i,j}\dd B^p_k(s)\\
   & +\frac{1}{2} \sum_{k=0}^\infty \sum^d_{i,j,k,l,p=1} \int_0^t \frac{\partial^2 \det(Dx(u,s))}{\partial A_{i,j}\partial A_{k,l}}\\
   \times& (Db^{\cdot,p}_k(\cdot,\mu_s,s)(x(u,s))Dx(u,s))_{i,j}(Db^{\cdot,p}_k(\cdot,\mu_s,s)(x(u,s))Dx(u,s))_{k,l} \dd s\\
   &= \int_0^t \dive(a(\cdot,\mu_s,s))(x(u,s))\det(Dx(u,s)) \dd s\\
   &+\sum_{k=0}^\infty \sum_{p=1}^d\int_0^t \dive(b^{\cdot,p}_k(\cdot,\mu_s,s)) \det(Dx(u,s)) \dd B^p_k(s)\\
   &+\frac{1}{2} \sum_{k=0}^\infty\int_0^t (\tr(Db^{\cdot,p}_k(\cdot,\mu_s,s)(x(u,s)))^2-\tr\left(\left(Db^{\cdot,p}_k(\cdot,\mu_s,s)(x(u,s))\right)^2\right) \det(Dx(u,s))\dd s
\end{align*}  
almost surely.
Hence $\left(\det(Dx(u,t))\right)_{t\ge 0}$ is the solution to a $1$-dim. linear equation and we can write down the solution explicitly
\begin{align*}
    \det(Dx(u,t))&= \exp\Bigg(\int_0^t \dive(a(\cdot,\mu_s,s))(x(u,s))-\frac{1}{2} \sum_{k=0}^\infty \sum_{p=1}^d\tr\left( \left(Db_k^{\cdot,p}(x(u,s),\mu_s,s)\right)^2 \right) \dd t \\
    &+ \sum_{k=0}^\infty \sum_{p=1}^d\int_0^t \dive(b_k^{\cdot,p}(\cdot,\mu_t,t))(x(u,t)) \dd B^p_k(t) \Bigg)  \fastsicher
\end{align*}

\end{beweis}

This representation of  $(\det(Dx(u,t)))_{u\in \Rd,t\ge0}$ enables us to control the $\lp{p}(\Rd)$-moments of $(p_t)_{t\ge 0}$ with respect to the initial data. 
\begin{lemma}
Under the assumptions \hyperlink{(A1)}{(A1)} and \hyperlink{(A2)}{(A2)} we have

\begin{align*}
    \sup_{u\in \Rd} \E\left(\sup_{0\le t\le T} \abs{\frac{1}{\det(Dx(u,t))}}^p\right)<\infty
\end{align*}
for all $u\in \Rd$, $T>0$ and $p\ge 1$.
\end{lemma}
\begin{beweis}
The statement follows essentially from Theorem \ref{Liouville}. Since we have that 
$\frac{1}{\det(Dx(u,s))}$ is the solution to the following linear SDE
\begin{align*}
    dZ(u,t) &= -\left(\dive(a(x(u,s),\mu_s,s)) +\frac{1}{2}\sum^\infty_{k=0} \sum_{p=1}^d \left(\tr(Db^{\cdot,p}_k(x(u,t),\mu_t,t))\right)^2-\tr(\left((Db^{\cdot,p}_k(x(u,t),\mu_t,t))^2\right)\right)\\
    \times& Z(u,t) \dd t - \sum_{k=0}^\infty \dive(b_k(\cdot,\mu_t,t))(x(u,t))Z(u,t) \dd B_k\\
    Z(u,0)= 1
\end{align*}
 then we get by the BDG inequality, 
 \begin{align*}
     \E(\sup_{0\le t\le T}\abs{Z(u,t)}^p) \le 3^{p-1}\Big(1+ \int_0^T\E(\abs{g(x(u,t),\mu_t,t)}^p\dd t \abs{Z(u,t)}^p)\Big)\le C(1+\int_0^T \E(\abs{Z(u,t)}^p)) \dd t
 \end{align*}
 where 
 \begin{align*}
     g(u,\mu,t)&= \dive(a(\cdot,\mu,t))(u) +\frac{3}{2}\sum^\infty_{k=0} \sum_{p=1}^d \left(\tr(Db^{\cdot,p}_k(\cdot,\mu,t)(u))\right)^2\\&-\sum^\infty_{k=0} \sum_{p=1}^d\tr(\left((Db^{\cdot,p}_k(\cdot,\mu,t)(u))^2\right) 
 \end{align*}
 from the assumption \hyperlink{(A2)}{(A2)} we know, that this function is bounded. Hence we get 
 \begin{align*}
     \E(\sup_{0\le t\le T}\abs{Z(u,t)}^p) \le C
 \end{align*}
 for all $u\in \Rd$ and $T>0$, which yields the claim.
\end{beweis}
With this result at hand we obtain  $p_t\in\lp{p}(\Rd)$  almost surely for all $t\ge 0$ whenever the initial data fulfills $p_0 \in \lp{p}(\Rd)$.
\begin{theorem}
    Assume \hyperlink{(A1)}{(A1)}, \hyperlink{(A2)}{(A2)} and \hyperlink{(A3)}{(A3)}, then for some constant $C>0$
    
    \begin{align*}
    \left( \E\left(\sup_{0\le t \le T} \int_{\Rd} p_t(u)^p \dd u\right)\right)^{\frac{1}{p}} \le C \norm{p_0}_{\lp{p}} 
    \end{align*}
   holds for all $T>0$ and $p\ge 1$.
 \end{theorem}
\begin{beweis}
By change of variables with the diffeomorphism $x(\cdot,t)$, we get 
\begin{align*}
    &\E\left(\sup_{0\le t \le T} \int_{\Rd} p_t(u)^p \dd u\right)\\
    &= \E\left(\sup_{0\le t \le T} \int_{\Rd} p_0(u)^p \frac{1}{\det(Dx(u,t))^{p-1}} \dd u\right) \\
    &\le \E\left(\int_{\Rd} p_0(u)^p \sup_{0\le t\le T} \frac{1}{\det(Dx(u,t))^{p-1}} \dd u\right)\\
    &\le \sup_{u\in \Rd}\E\left(\sup_{0\le t\le T} \frac{1}{\det(Dx(u,t))^{p-1}}\right) \int_{\Rd} p_0(u)^p \dd u = C \norm{p_0}^p_{\lp{p}(\Rd)}
\end{align*}
\end{beweis}
\begin{bemerkung}
The Lebesgue densities are solutions to SPDEs of Fokker-Planck type. Namely for all $\phi \in \ccc_c^2(\Rd)$, we have 
\begin{align*}
    \int_{\Rd }\phi(u)p_t(u) \dd u &= \int_{\Rd} \phi(x(u,t)) p_0(u) \dd u \\
    &= \int_{\Rd} \phi(u) p_0(u) \dd u + \\
     &= \sum_{i=1}^d\int_{\Rd} \int_0^t\partial_i \phi(x(u,s)) a^i(x(u,s),p_s,s)  \dd s   p_0(u) \dd \\
     &\sum_{k=0}^\infty\sum_{i,p=1}^d \int_0^t\partial_i \phi(x(u,s)) b^{i,p}_k(x(u,s),p_s,s) \dd B^p_k(s) \\
     &+\frac{1}{2} \sum_{i,j=1}^d \sum_{k=0}^\infty \int_0^t \partial_i\partial_j \phi(x(u,s)) b^{j,p}_k(x(u,s),p_s,s)b^{i,p}_k(x(u,s),p_s,s) \dd s p_0(u) \dd u\\
     &= \int_{\Rd} \phi(u) p_0(u) \dd u \\
     &+ \sum_{i=1}^d \int_0^t \int_{\Rd} \partial_i\phi(u) a^i(u,p_s,s) \dd u \dd s \\
     &\sum_{i,p=1}^d \sum_{k=0}^\infty\int_0^t \int_{\Rd} \partial_i\phi(u) b_k^{i,p}(u,p_s,s) \dd u \dd B^p_k(s) \\
     &\frac{1}{2}\sum_{i,j,p=1}^d \sum_{k=0}^\infty \int^t_{0}  \int_{\Rd}\partial_i\partial_j\phi(u)b_k^{i,p}(u,p_s,s)b_k^{j,p}(u,p_s,s)  \dd u\dd s
\end{align*}
 one can actually show, under some additional assumptions that this is the unique mass conservative and positive solution to the SPDE (i.e. $\int_{\Rd} p_t(u) \dd u=1$ and $p_t\ge0$  for all $t\ge 0$) for more details see \cite{https://doi.org/10.48550/arxiv.2207.05705}, for more SPDEs of this type see \cite{kotelenez2008stochastic}.
\end{bemerkung}


\section{Intermittency}
Now that we have collected the most important properties of $(p_t)_{t\ge 0}$ we may define the intermittency property mathematically. 
\begin{assumption}\raisebox{\ht\strutbox}{\hypertarget{(A4)}{}}
The Lebesgue density $p_0$ of the initial condition $\mu_0$ suffices $p_0\in \lp{p}(\Rd)$ for all $p\ge 1$ 
\end{assumption}
\begin{definition}[Intermittency]\label{DefIntermittency}
The random field $(p_t)_{t\ge 0}$ is intermittent, if 
\begin{align*}
    \lim_{t\to \infty} \frac{\ln(\int_{\Rd} p_t(u)^p \dd u)}{t} := \lambda_p
\end{align*}
exists and 
\begin{align*}
    \left(\frac{\lambda_p}{p}\right)_{p\ge 1}
\end{align*}
is strictly increasing.  
\end{definition}



The definition of the intermittency property is related to the moment Lyapunov exponents introduced by Khasminskii \cite{khasminskii2011stochastic}. These were used in the investigation of intermittency phenomena for the parabolic Anderson model (e.g. \cite{carmona1995stationary}, \cite{khoshnevisan2014analysis}). The implication of this behaviour, is that higher moments of the quantity $(p_t)_{t\ge 0}$ dominate lower moments as $t\to \infty$ i.e.
\begin{align*}
  \lim_{t\to \infty}  \frac{\norm{p_t}_{\lp{p}(\Rd)}}{\norm{p_t}_{\lp{q}(\Rd)}} = 0
\end{align*}
such behaviour usually occurs, whenever such a quantity converges to a function with very high peaks on small sets, or some function with close behaviour to the $\delta$ Dirac functional. 
\begin{beispiel}\label{BspIntermittenz}
Consider the following linear SDE with interaction 
\begin{align*}
dx(u,t) &= \int_{\Rd}A(v-x(u,t))\mu_t(\dd v) \dd t \sum_{p=1}^d \Sigma_p(v- x(u,t)) \mu_t(\dd v)\dd B^p(t)\\
\forall u\in \Rd \quad x(u,0) &= u 
\end{align*}
We can immediately compute the determinant of the derivative of $x(\cdot,t)$, with Lemma \ref{Liouville}
\begin{align*}
\det(Dx(u,t))= \exp\left(t(\tr(A) -\sum_{p=1}^d \tr(\Sigma^2_p)) + \sum^d_{p=1}\tr(\Sigma_p) B_p(t) \right)
\end{align*}
hence we obtain from the transformation rule under diffeomorphisms 
\begin{align*}
    \int_{\Rd} p_t(u)^p \dd u = \exp\left(-(p-1)\left(t(\tr(A)\sum_{p=1}^d \tr(\Sigma^2_p)) + \sum^d_{p=1}\tr(\Sigma_p) B_p(t)\right)\right) \norm{p_0}^p_{\lp{p}(\Rd)}.
 \end{align*}
 Thus we can compute with the strong law of large numbers for the Brownian motion
 \begin{align*}
     \lambda_p = \lim_{t\to \infty} \frac{-(p-1)\left( t(\tr(A)\sum_{p=1}^d \tr(\Sigma^2_p)) + \sum^d_{p=1}\tr(\Sigma_p) B_p(t)  \right)}{t} =  -(p-1)\left(\tr(A)-\frac{1}{2} \sum_{p=1}\tr(\Sigma^2_p)\right). 
 \end{align*}
 We finally conclude that in this case $(p_t)_{t\ge 0}$ is intermittent if and only if 
 \begin{align*}
     \tr(A)-\frac{1}{2}\sum_{p=1}^d \tr(\Sigma^2_p)< 0
 \end{align*}
\end{beispiel}

The structure of the problem suggests, that it can be tackled by considering the asymptotics, of $\dett(Dx(\cdot,t))$ only. In order to do so it is natural to investigate Lyapunov exponents. In order to get a similar situation, as in Example \ref{BspIntermittenz} we will consider Lyapunov exponents

\begin{definition}The limit
\begin{align*}
    \lim_{t\to \infty} \frac{\ln(\det(Dx(u,t)))}{t} = \lambda(u)\fastsicher
\end{align*}
    is called Lyapunov exponent. 
\end{definition}
\begin{bemerkung}
    Note that the main difficulty in guaranteeing existence of Lyapunov exponents, lies within the fact, that SDEs with interaction are nonautonomous with respect to $x(u,t)$. Thus we can not simply apply techniques, which are well known in the theory of random dynamical systems. However even if $(x_\mu(\cdot,t),\mu_t)$, where $\mu$ in the index denotes the initial measure,  forms a random dynamical system, to which one could try to apply random dynamical system approaches. This path turns out to be not as successful as the one presented in this article.
\end{bemerkung}
 The Lyapunov exponent yields us 
 \begin{align}\label{Lyapdet}
     \det(Dx(u,t))\approx \exp(t \lambda(u))
 \end{align}

 for $t\gg 1$. With a method of steepest descent approach we thus get 

\begin{align*}
    \lim_{t\to \infty} \frac{\ln\left(\int_{\Rd} p_0(u)^p \exp(-(p-1)t\lambda (u)) \right)}{t} = \sup_{u\in K} -\lambda(u) (p-1).
\end{align*}
if $\supp(p_0)\subset K $ for a compact $K\subset \Rd$. We will first justify \eqref{Lyapdet} under certain conditions.   
\begin{assumption}\raisebox{\ht\strutbox}{\hypertarget{(A5)}{}}
    Let $a$ and $(b_k)_{k\ge 0}$ be such that there exists a function $\phi$ and $\alpha>0$ with 
    \begin{align*}
        \skalar{u-v,\phi(u)-\phi(v)} \le -\alpha \abs{u-v}^2
    \end{align*}
    for all $u,v,\in \Rd$ and $(b_k)_{k\ge 0}$ such that $2\alpha -B^2(4q-1) >0$ where $B$ is the Lipschitz constant of $(b_k)_{k\ge 0}$ with respect to the $l^2$-norm and for some $q>\max\{d,\delta\}$, if $(b_k)_{k\ge0}$ is not bounded in the $l^2$-sense and $2\alpha -B^2(2q-1)>0$ if $(b_k)_{k\ge 0}$ is bounded. Furthermore let 
    \begin{align*}
        a(u,\mu) = \int_{\Rd} \phi(u-v) \mu(\dd v)
    \end{align*}
    for all $u\in \Rd,\mu\in \M$, we will assume further, that $\phi$ is differentiable with bounded, $\delta$-Hölder derivative.
 \end{assumption}

 \begin{bemerkung}
     As we will later prove rigorously,  Assumption \hyperlink{(A5)}{(A5)} guarantees, that
     \begin{align*}
         \lim_{t\to \infty}\abs{x(u,t)-x(v,t)} =0 \fastsicher
     \end{align*}
     for all $u,v\in \Rd$.
     In other words, the trajectories from two different starting points approach each other almost surely, as $t\to \infty$.
 \end{bemerkung}
We will not comment on the case where $(b_k)_{k\ge 0}$ is bounded. Since this case distinction is only relevant in the proof of Lemma \ref{BVArzela}
 \begin{lemma}\label{Verschiebungskompakt}
     Under the assumptions \hyperlink{(A1)}{(A1)} and \hyperlink{(A5)}{(A5)}, if $p\ge 1$ fulfills 
     \begin{align*}
         -2\alpha + B^2(2p-1) <0 
     \end{align*}
     then for all $t\ge 0$ and $u,v\in \Rd$, we get
     \begin{align*}
         \E(\abs{x(u,t)-x(v,t)}^{2p})\le \abs{u-v}^{2p}
     \end{align*} 
 \end{lemma}
 \begin{beweis}
     Note first, that by Itô's formula, we have 
          \begin{align*}
     \begin{split}
         \E(\abs{x(u,t)-x(v,t)}^{2p})&= \abs{u-v}^{2p} \\&+ \E\Big(  2p \int_0^t \int_{\Rd} \abs{x(u,s)-x(v,s)}^{2p-2} \\
         &\int_{\Rd} \skalar{x(u,s)-x(v,s),\phi(x(u,s)-x(r,s))-\phi(x(v,s)-x(r,s))} \mu(\dd r) \dd s \Big)\\
         &+ 2p\int_0^t \abs{x(u,s)-x(v,s)}^{2p-2} \sum_{k=0}^\infty \norm{b_k(x(u,s),\mu_s)-b_k(x(v,s),\mu_s)}^2_{HS} \dd s\\
         &+2p(p-1) \sumkinfty\sum_{l=1}^d \int^t_{0} 
        \abs{x(u,s)-x(v,s)}^{2p-4} \\
         \times&\skalar{x(u,s)-x(v,s),b^{\cdot,l}_k(x(u,s),\mu_s)-b^{\cdot,l}_k(x(v,s),\mu_s)}^2 \dd s \\
         &\le\abs{u-v}^{2p} +-2p\alpha + pB^2 +2p(p-1) B^2 \le \abs{u-v}^{2p}
         \end{split}
     \end{align*}
 \end{beweis}
\begin{lemma}\label{Lyapneg}
Assume \hyperlink{(A1)}{(A1)},\hyperlink{(A2)}{(A2)} and \hyperlink{(A5)}{(A5)}. Then we have for all $u\in \Rd$: 
\begin{align*}
   \limsup_{t\to \infty} \frac{\ln(\det(Dx(u,t)))}{t} < 0 \fastsicher
\end{align*}

\end{lemma}
\begin{beweis}
We will first show, that the assumptions imply for all $u\in \Rd$ and $\mu\in \M$
\begin{align*}
    \dive(a(\cdot,\mu))(u) \le -d\alpha.
\end{align*}
Let $j=1,\dots,d$ and $u,v\in \Rd$ such that $u_j\neq v_j$. We get 

\begin{align*}
    \sum_{i=1}^d (u_i-v_i)(\phi^i(u)-\phi^i(v) ) \le -\alpha \sum_{i=1}^d (u_i-v_i)^2 \le -\alpha (u_j- v_j)^2     
\end{align*}
By Letting $u_i \to v_i$ for all $i\neq j$, this implies
\begin{align*}
    \frac{\phi^{j}(u)-\phi^j(v)}{u_j-v_j} \le -\alpha
\end{align*}
which finally yields 
\begin{align*}
    \dive(\phi(\cdot))(u)\le -d\alpha
\end{align*} implying 
\begin{align*}
    \dive(a(\cdot,\mu))(u)\le -d\alpha.
\end{align*} 
We will proceed by showing that
\begin{align*}
-\sum_{p=1}^d\sum_{k=0}^\infty \tr((Db^{\cdot,p}_k)^2)\le dB^2.    
\end{align*} First we need to compute the optimal Lipschitz constant of $(b_k)_{k\ge 0}$ in the $l^2$-sense. Consider for $u,v\in \Rd$ and $\mu\in \M$
\begin{align*}
   & \sum_{p=1}^d\sum_{k=0}^\infty \abs{b_k^{\cdot,p}(u,\mu)-b_k^{\cdot,p}(v,\mu)}^2\\ 
   =& \sum_{p=1}^d \sum_{k=0}^\infty \abs{\int_0^1 Db^{\cdot,p}_k(u+t(v-u),\mu)\dd t(u-v)}^2 \\
   \le&\abs{u-v}^2\sum^d_{p=1} \sum_{k=0}^\infty \norm{\int_0^1 Db^{\cdot,p}_k(u+t(v-u),\mu)\dd t}^2_{OP} \\
\le& \abs{u-v}^2\sup_{u,v\in \Rd} \sumkinfty \sum_{p=1}^d \norm{\int_0^1 Db^{\cdot,p}_k(u+t(v-u),\mu)\dd t}^2_{OP} \\
   =& \abs{u-v}^2 \sup_{u\in \Rd} \sum^d_{p=1} \sum_{k=0}^\infty \norm{Db^{\cdot,p}_k(u,\mu)}^2_{OP}
\end{align*}
where $\norm{\cdot}_{OP}$ denotes the operator norm. 
The last equality can be shown in the following way, for all $u\in \Rd$ and $\mu\in \M$ we have 
\begin{align*}
   &\sumkinfty\sum_{p=1}^d \norm{Db_k^{\cdot,p}(u,\mu)}^2_{OP}\\
   =& \sumkinfty\sum_{p=1}^d \norm{\int_0^1 Db_k^{\cdot,p}(u+t(u-u),\mu) \dd t}^2_{OP}\\
   \le& \sup_{u,v\in \Rd}\sumkinfty \sum_{p=1}^d \norm{\int_0^1Db_k^{\cdot,p}(u+t(u-v),\mu) \dd t }^2_{OP}
\end{align*}
thus $"\ge"$ follows, for $"\le"$ observe that for all $u,v\in \Rd$ and $\mu\in \M$ we have
\begin{align*}
    &\sumkinfty \sum_{p=1}^d \norm{\int_0^1Db_k^{\cdot,p}(u+t(u-v),\mu) \dd t }^2_{OP}\\
    \le& \sumkinfty\sum_{p=1}^d \int_0^1 \norm{Db_k^{\cdot,p}(u+t(u-v),\mu)}^2_{OP} \dd t\\
    =& \int_0^1 \sumkinfty\sum_{p=1}^d  \norm{Db_k^{\cdot,p}(u+t(u-v),\mu)}^2_{OP} \dd t \\
    \le & \sup_{u\in \Rd} \sumkinfty \sum_{p=1}^d \norm{Db_k^{\cdot,p}(u,\mu)}^2. 
\end{align*}
Hence, we can finally conclude
\begin{align*}
    B^2 = \sup_{u\in \Rd,\mu\in \M} \sum^d_{p=1} \sum_{k=0}^\infty \norm{Db^{\cdot,p}_k(u,\mu)}^2_{OP}.
\end{align*}

Now we have, for all $u\in \Rd,\mu\in\M$
\begin{align*}
    &\abs{\sum^d_{p=1}\sum_{k=0}^\infty\tr((Db^{\cdot,p}_k(u,\mu))^2)}\\
    &\le \abs{\sum^d_{p=1}\sum_{k=0}^\infty \sum^d_{i,j=d} \left(Db^{\cdot,p}_k(u,\mu)\right)_{i,j}\left(Db^{\cdot,p}_k(\cdot,\mu)(u)\right)_{j,i}}\\
    &\le \sum_{k=0}^\infty  \sum_{p=1}^d \norm{Db^{\cdot,p}_k(\cdot,\mu)(u)}^2_{HS}\\
    &\le dB^2
\end{align*}
where we used $\norm{A}_{HS}\le \sqrt{d}\norm{A}_{OP}$ for all $A\in \R^{d \times d}$. Therefore 
\begin{align*}
    -\frac{1}{2} \sum_{p=1}^d \sum_{k=0}^\infty \tr(\left(Db^{\cdot,p}_k(u,\mu)\right)^2)\le d\frac{B^2}{2}
\end{align*}
for all $u\in \Rd,\mu\in \M$. 
Finally we arrive at 
\begin{align*}
   & \limsup_{t\to \infty} \frac{\ln(\det Dx(u,t))}{t} \\ 
   &\le \limsup_{t\to \infty} \frac{1}{t} \int_0^t \left(\dive(a) -\frac{1}{2}\sum^d_{p=1}\sum_{k=0}^\infty \tr((Db^{\cdot,p}_k)^{2})\right)(x(u,t),\mu_t) \\
   &+\limsup_{t\to \infty} \frac{1}{t} \sum_{k=0}^\infty \sum_{p=1}^d\int_0^t \dive(b^{\cdot,p}_k)(x(u,s),\mu_s) \dd B^p_k(s)\\
   &\le \limsup_{t \to \infty}  \frac{1}{t} \int_0^t d(-\alpha+\frac{B^2}{2}) \dd s < 0 \fastsicher
 \end{align*}

where we used the strong law of large numbers for martingales (Theorem 9, p.142 \cite{liptser2012theory}) in the last estimate. Which can be applied since 
\begin{align*}
    \int_0^\infty \frac{\sumkinfty\sum_{p=1}^d \left(\dive(b_k^{\cdot,p}(\cdot,\mu_t))(x(u,t))\right)^2}{(1+t)^2} \dd t \le C\int_0^\infty \frac{1}{(1+t)^2} \dd t<\infty 
\end{align*}
since the second term is bounded by $\sup_{u,\in \Rd, \mu\in \M} \sum_{p=1}^d \sum_{k=0}^\infty \dive(b^{\cdot,p}_k(u,\mu))^2$ which is finite by \hyperlink{(A2)}{(A2)}.
\end{beweis}
\begin{bemerkung}
    We will show in Lemma \ref{MartglmKonvfs} that the null set which is a result of Lemma \ref{Lyapneg} does not depend on $u\in \Rd$.
 \end{bemerkung}
 We want to show that the Lyapunov exponents are relatively compact in $C(K)$ lamost surely for all $K\subset \Rd$.  

\begin{notation}
We will write 
\begin{align*}
    \frac{\ln(\det(Dx(u,t))}{t} &= f_t(u) = \frac{BV_t(u)}{t} + \frac{M_t(u)}{t}\\
    &=\frac{\int_0^t \dive(a(\cdot,\mu_s))(x(u,s))-\frac{1}{2}\sum_{p=1}^d \tr\left(\left(Db^{\cdot,p}_k(\cdot,\mu_s)(x(u,s))\right)^2\right)\dd s}{t}\\
    &+\frac{\sum_{k=0}^\infty \sum_{p=1}^d \int_0^t \dive(b^{\cdot,p}_k(\cdot,\mu_s))(x(u,s))\dd B^p_k(s)}{t}
\end{align*} 
where $BV$ stands for the bounded variation part and $M$ for the martingale part.
\end{notation}

\begin{lemma}\label{HölderBV}
Under \hyperlink{(A2)}{(A2)}, we have for all $u,v\in \Rd$
\begin{align*}
    \frac{\abs{BV_t(u)-BV_t(v)}}{t} \le C \frac{\int_0^t \abs{x(u,s)-x(v,s)}^\delta \dd s}{t} \fastsicher .
\end{align*}
 
\end{lemma}

\begin{beweis}
    It suffices to show that $\sum_{p=1}^d \sum_{k=0}^\infty \tr((Db^{\cdot,p}_k)^2)$ is Hölder continuous for all $\mu\in \M$, this combined with \hyperlink{(A2)}{(A2)} we get the result immediately from the definition of $BV_t$ for $t\ge 0$.   
    \begin{align*}
      &\abs{\sum_{p=1}^d \sum_{k=0}^\infty \tr((Db^{\cdot,p}_k)(u,\mu)^2)-\sum_{p=1}^d \sum_{k=0}^\infty \tr((Db^{\cdot,p}_k)(v,\mu)^2)}  \\
      &\le \sum_{k=0}^\infty \sum^d_{p=1} \sum_{i,j=1}^d \abs{\frac{\partial b^{i,p}_k}{\partial u_j}\frac{\partial b^{j,p}_k}{\partial u_i}(u,\mu)- \frac{\partial b^{i,p}_k}{\partial u_j}\frac{\partial b^{j,p}_k}{\partial u_i}(v,\mu)}\\
      &\le \sum_{k=0}^\infty \sum^d_{p=1} \sum_{i,j=1}^d \abs{\frac{\partial b^{i,p}_k}{\partial u_j}(u,\mu)- \frac{\partial b^{i,p}_k}{\partial u_j}(v,\mu)}\abs{\frac{\partial b^{j,p}_k}{\partial u_i}(u,\mu)}\\
      &+\sum_{k=0}^\infty \sum^d_{p=1} \sum_{i,j=1}^d \abs{\frac{\partial b^{j,p}_k}{\partial u_i}(u,\mu)- \frac{\partial b^{j,p}_k}{\partial u_i}(v,\mu)}\abs{\frac{\partial b^{i,p}_k}{\partial u_j}(v,\mu)}\\
      &\le 2 \sup_{u\in \Rd}\left(\sumkinfty \norm{Db_k(u,\mu)}^2_{HS} \right)^{\frac{1}{2}} \\
      &\times 2 \left(\sumkinfty \norm{Db_k(u,\mu)-Db_k(v,\mu)}^2_{HS}  \right)^{\frac{1}{2}}\\
      &  \le C\abs{u-v}^\delta      \fastsicher
     \end{align*} 
\end{beweis}

\begin{lemma}\label{MartglmKonvfs}
    Under the assumptions \hyperlink{(A1)}{(A1)},\hyperlink{(A2)}{(A2)} and \hyperlink{(A5)}{(A5)} we obtain

    \begin{align*}
        \lim_{t\to \infty} \sup_{u\in K} \frac{M_t(u)}{t}=0 \fastsicher
    \end{align*}
    for all compact $K\subset \Rd$.
\end{lemma}

\begin{beweis} Consider w.l.o.g $K=[0,1]^d$.
     Let $0<\delta\le 1$ be the Hölder constant from \hyperlink{(A2)}{(A2)}. Then for all $T>0$  we have by the Burkholder-Davis-Gundy inequality 
     \begin{align}\label{absch1}
     \begin{split}
         &\E\left(\left(\frac{\sup_{0\le t\le T }\abs{M_t(u)-M_t(v)}}{T}\right)^{\frac{2q}{\delta}}\right)
         \\\le& \frac{1}{T^{\frac{p}{\delta}}} \E\left(\sumkinfty\sum_{p=1}^d \left(\int_0^T \left(\dive(b^{\cdot,p}_k)(x(u,s),\mu_s)-\dive(b^{\cdot,p}_k)(x(v,s),\mu_s)\right)^2\dd s\right)^{\frac{q}{\delta}} \right) \\
         &\le C\frac{T^{\frac{q}{\delta}-1}}{T^{\frac{2q}{\delta}}} \int_0^T \E(\abs{x(u,s)-x(v,s)}^{2q}) \dd s. 
         \end{split}
     \end{align}

     To estimate the last term, note that Lemma \ref{Verschiebungskompakt}
    
       implies 
     \begin{align*}
         \E(\abs{x(u,t)-x(v,t)}^{2q})\le \abs{u-v}^{2q} 
     \end{align*}
    by the Hölder inequality.
     Hence we can estimate the last term in \eqref{absch1} by 
     \begin{align*}
         \frac{C}{T^{\frac{q}{\delta}}} \abs{u-v}^2q 
     \end{align*}
     for all $u,v\in [0,1]^d$.Thus from Lemma 1.8.1 and Lemma 1.8.2 in \cite{kunita2019stochastic}, we know that 
     \begin{align*}
         \E\left(\sup_{u\in [0,1]^d}\left(\frac{\sup_{0\le t\le T}\abs{M_t(u)-M_t(0)}}{T}\right)^{\frac{2q}{\delta}}\right)\le C\E(L^{\frac{2q}{\delta}}) \sup_{u\in [0,1]^d } \norm{u}^\beta_{\infty} \le C \frac{1}{T^{\frac{q}{\delta}}} 
     \end{align*}
     Here $L$ is the local Hölder constant coming from Lemma 1.8.1 \cite{kunita2019stochastic} and $\beta$ the corresponding Hölder exponent.
 From this the inequality can easily be extended to $[-n,n]^d$ for all $n\in \N$ by only changing the constant $C>0$, showing the claim. Furthermore we have for all $T>0$:
     \begin{align*}
         &\E(\left(\frac{\sup_{0\le t\le T} \abs{M_t(0)}}{T}\right)^{\frac{2q}{\delta}} )\\
         \le& C \frac{1}{T^{\frac{2q}{\delta}}}\E\left( \left(\int_0^T\sumkinfty\sum_{p=1}^d (\dive(b_k^{\cdot,p}(\cdot,\mu_t))(x(0,t)))^2 \dd t\right)^{\frac{q}{\delta}} \right)\\
         \le& C \frac{1}{T^{\frac{q}{\delta}}} 
     \end{align*}
     We will now show almost sure convergence. First of all we can assume without loss of generality that $0\in K$, since otherwise we can just enlarge $K$. 
     Let $\eps>0$ and consider for $n\in \N$ 
\begin{align*}
    A_n:=\left\{ \abs{\sup_{0\le t\le n}\sup_{u\in K} M_t(u)}>\eps n  \right\}
\end{align*}
then we can show
\begin{align*}
    \W(A_n)&\le (n\eps)^{-\frac{2q}{\delta}} \E\left((\sup_{u\in K} \sup_{0\le t\le n}\abs{M_t(u)})^{\frac{2q}{\delta}}\right)\\
&\le C (n\eps)^{-\frac{2q}{\delta}}\left(\E\left(\left(\sup_{u\in K }\sup_{0\le t \le n} \abs{M_t(u)-M_t(v)}\right)^{\frac{2q}{\delta}} \right)+ \E\left(\left(\sup_{0\le t\le n} \abs{M_t(0)}\right)^{\frac{2q}{\delta}} \right)\right) \\
    &\le  C \frac{1}{n^{\frac{q}{\delta}}}.
\end{align*}
By the choice of $q$, we obtain the summability of the right-hand side and thus $A_n$ can almost surely occur for only a finite number of $n\in N$, by the Borel-Cantelli lemma. Now for  $t\gg 1$ take $n\in \N$ such that $n-1\le t\le n$, then we get 
\begin{align*}
    \frac{\abs{\sup_{u\in K}M_t(u)}}{t}&= 
    \frac{\abs{\sup_{u\in K}M_t(u)}}{n} \frac{n}{t} \\
    &\le \frac{\abs{\sup_{u\in K}M_t(u)}}{n} \frac{n}{n-1} \le C\eps  \fastsicher
\end{align*}
where $C=\sup_{n\ge 2} \frac{n}{n-1}$. Choosing a sequence $\eps_n\searrow 0$, we obtain  
\begin{align*}
    \limsup_{t\to \infty} \frac{\abs{\sup_{u\in K} M_t(u)}}{t} = 0
\end{align*} almost surely, which yields the claim.
\end{beweis}
We will proceed, by treating the part of bounded variation
\begin{lemma}\label{BVArzela}
    Under the assumptions \hyperlink{(A1)}{(A1)}, \hyperlink{(A2)}{(A2)} and \hyperlink{(A5)}{(A5)}, for all compact $K\subset \Rd$ there exists a set $N$ such that $\W(N)=0$ and for all $\omega \in \Omega\setminus N$ there is some $T(\omega,K)>0$ such that 
    \begin{align*}
        \left(\frac{BV_t(\omega)}{t}\right)_{t\ge T(\omega,K)} 
    \end{align*}
    is relatively compact in $\ccc(K)$.  
\end{lemma}
\begin{beweis}
 Observe that by Itô's formula, we have for all $u,v\in \Rd$ 
    \begin{align*}
        \abs{x(u,t)-x(v,t)}^{2} &= \abs{u-v}^{2} \\
        +& 2\int_0^t \skalar{x(u,s)-x(v,s), \dd x(u,s)-x(v,s)} \\
        & 2\int_0^t \sumkinfty \norm{b_k(u,\mu_s)-b_k(v,\mu_s)}_{HS}^{2}\dd \skalarq{x(u,\cdot)-x(v,\cdot)}_s \\
        &= \abs{u-v}^{2} \\
        &+ 2\int_0^t \int_{\Rd} \skalar{x(u,s)-x(v,s), \phi(x(u,s)-x(r,s))-\phi(x(v,s)-x(r,s))  } \mu(\dd r) \dd s \\
        &+ \sumkinfty \sum_{p=1}^d2\int_0^t \skalar{x(u,s)-x(v,s),b_k^{\cdot,p}(x(u,s),\mu_s)-b^{\cdot,p}_k(x(v,s),\mu_s)} \dd B_k^p(s) \\
        &+\sumkinfty\sum^d_{p=1}\int_0^t \norm{b_k(x(u,s),\mu_s)-b_k(x(v,s),\mu_s)}_{HS}^{2} \dd s.
        \\ &\le \abs{u-v}^2 + (-2\alpha - B^2) \int_0^t \abs{x(u,s)-x(v,s)}^2 \dd s \\& +   \sum_{p=1}^d \sumkinfty 2\int_0^t \skalar{x(u,s)-x(v,s),b_k^{\cdot,p}(x(u,s),\mu_s)-b_k^{\cdot,p}(x(v,s),\mu_s)} \dd B_k^p(s)   \fastsicher 
    \end{align*}
    Note that this holds for all $u,v\in \Rd$ almost surely, hence it holds for all $\omega\in \Omega\setminus N_1$ where $N_1$ is the null set on which all the terms above are discontinuous in $(t,u)\in \R_+\times \Rd$. 
Grönwall's Lemma yields us for all 
\begin{align}\begin{split}\label{Groenwall}
 \abs{x(u,t)-x(v,t)}^2 &\le \exp(-(2\alpha-B^2)t) \times\Big( \abs{u-v}^2 \\&+ 2\sup_{0\le s\le t} \sup_{u,v\in K} \sumkinfty \sum^d_{p=1}\int_0^s \skalar{x(u,r)-x(v,r),b^{\cdot,p}_k(x(u,r),\mu_r)-b_k^{\cdot,p}(x(v,r),\mu_r)} \dd B_k^p(r)\Big). 
 \end{split}
\end{align}
almost surely.
Now set 
\begin{align*}
    Y_t =\sup_{u,v\in K}2\sup_{0\le s\le t}  \sumkinfty \sum^d_{p=1}\int_0^s \skalar{x(u,r)-x(v,r),b^{\cdot,p}_k(x(u,s),\mu_s)-b_k^{\cdot,p}(x(v,s),\mu_s)} \dd B_k^p(r)
\end{align*}
We will show in a similar way, as in Lemma \ref{MartglmKonvfs}, that
\begin{align}\label{f.s.Konv}
\lim_{t\to \infty}\frac{Y_t}{t}= 0 \fastsicher
\end{align}
To see this, consider 
\begin{align*}
    Y_t =: \sup_{u,v\in K} \sup_{0\le s \le t} X_s(u,v)
\end{align*}
we will show 
\begin{align*}
    \E\left(\left(\sup_{0\le r \le t} \abs{X_r(u_1,u_2)-X_r(v_1,v_2)}\right)^{2q}\right)\le C(K,b,d)t^{q}\abs{u-v}^{2q}
\end{align*}
where $u=(u_1,u_2),v=(v_1,v_2)\in K^2$, for all $q\ge 1$. W.l.o.g consider $K=[0,1]^d$. Observe first 
\begin{align}\label{SPlokLip}
    \abs{\skalar{x,y}-\skalar{v,w}}\le \abs{\skalar{x-v,y}}+ \abs{v,y-w}\le \abs{y} \abs{x-v} + \abs{v}\abs{y-w} .
\end{align}
where here $\skalar{\cdot,\cdot}$ denotes the inner product on $\Rd$.
From the Burkholder-Davis-Gundy inequality we now get 
 \begin{align*}
       & \E\left(\left(\sup_{0\le r \le t} \abs{X_r(u_1,u_2)-X_r(v_1,v_2)} \right)^{2q}\right) \\
       & \le\E\Bigg(\bigg(2\sup_{0\le r\le t}\Big\vert\sum^d_{p=1}\sumkinfty\int_0^r \skalar{x(u_1,s)-x(u_2,s),b_k^{\cdot,p}(x(u_1,s),\mu_s)-b^{\cdot,p}_k(x(u_2,s),\mu_s)}\\
       &-\skalar{x(v_1,s)-x(v_2,s),b_k^{\cdot,p}(x(v_1),\mu_s)-b_k^{\cdot,p}(x(v_2,s),\mu_s)}\dd B_k^p(s)\Big\vert   \bigg)^{2q} \Bigg)\\
       &\le \E\Bigg(2\Bigg(\int_0^t \Big(\sumkinfty \sum_{p=1}^d\skalar{x(u_1,s)-x(u_2,s),b_k^{\cdot,p}(x(u_1,s),\mu_s)-b^{\cdot,p}_k(x(u_2,s),\mu_s)}\\
       &-\skalar{x(v_1,s)-x(v_2,s),b_k^{\cdot,p}(x(v_1),\mu_s)-b_k^{\cdot,p}(x(v_2,s),\mu_s)}\Big)^2\dd s   \Bigg)^q \Bigg)\\
       &\le C(d,q) t^{q-1} \E\Bigg(\bigg(\int_0^t \big(\sumkinfty\sum_{p=1}^d\skalar{x(u_1,s)-x(u_2,s),b_k^{\cdot,p}(x(u_1,s),\mu_s)-b_k^{\cdot,p}(x(u_2,s),\mu_s)}\\
       &-\skalar{x(v_1,s)-x(v_2,s),b^{\cdot,p}(x(v_1),\mu_s)-b^{\cdot,p}(x(v_2,s),\mu_s)}^2\big)^{q}\dd s   \bigg) \Bigg)\\
       &\overset{\eqref{SPlokLip}}{\le} C(d,q) t^{q-1} \E\Bigg(\bigg(\int_0^t \abs{x(u_1,s)-x(v_1,s)-(x(u_2,s)-x(v_2,s))}^{2q}\\
      & \times\big(\sum^d_{p=1}\sumkinfty\abs{b_k^{\cdot,p}(x(u_1,s),\mu_s)-b_k^{\cdot,p}(x(v_1,s),\mu_s)}\big)^{2q}\\
       &+\abs{x(u_2,s)-x(v_2,s)}^{2q}\\
       &\times\big(\sum_{p=1}^d\sumkinfty\abs{b_k^{\cdot,p}(x(u_1,s),\mu_s)-b_k^{\cdot,p}(x(v_2,s),\mu_s)-(b^{\cdot,p}_k(x(u_1,s),\mu_s)-b^{\cdot,p}_k(x(v_1,s),\mu_s))}^2\big)^{q}\dd s   \bigg) \Bigg)\\
       \overset{\text{Cauchy-Schwarz},\hyperlink{(A1)}{(A1)}}&{\le} C(d,q) t^{q-1} \int_0^t \Bigg(\Big(\E(\abs{x(u_1,s)-x(v_1,s)}^{4q}) +\E(\abs{x(u_2,s)-x(v_2,s)}^{4q})\Big) \\
       &\times \E(\abs{x(u_1,s)-x(v_1,s)}^{4q}) \Bigg)^{\frac{1}{2}} \\
       &+ \Bigg(\E(\abs{x(u_2,s)-x(v_2)}^{4q})\\
       &\times \Big(\E(\abs{x(u_1,s)-x(v_1,s)}^{4q})+\E(\abs{x(u_2,s)-x(v_2,s)}^{4q})\Big)\Bigg)^{\frac{1}{2}} \dd s \\
       \overset{\text{Lemma} \ref{Verschiebungskompakt}}&{\le} C(d,q) t^{q-1} \\
       &\times \int_0^t \Big(\bigg(\abs{u_1-v_1}^{4q} +\abs{u_2-v_2}^{4q}\bigg)\abs{u_1-v_1}^{4q}\Big)^{\frac{1}{2}} \\
       &+ \Big(\bigg(\abs{u_1-v_1}^{4q} +\abs{u_2-v_2}^{4q}\bigg)\abs{u_2-v_2}^{2q}\Big)^{\frac{1}{2}} \dd s \\
       &\le C(d,q,K) t^{q} (\abs{u_1-v_1}^{2q}+\abs{u_2-v_2}^{2q})\\
       &\le C(d,q,K) t^{} \abs{(u_1,u_2)-(v_1,v_2)}^{2q}
    \end{align*}
    for all $u_i,v_i\in K$ where $i=1,2$ and $u=(u_1,u_2),v=(v_1,v_2)$. Observe that in the case, where $(b_k)_{k\ge 0}$ is bounded, we can use the Hölder inequality for the exponents $1$ and $\infty$, instead of the Cauchy-Schwarz inequality, to obtain the same result. Note that the constants appearing may change with each inequality.  Since $[0,1]^{2d}= [0,1]^d \times [0,1]^d$ we can consider  
    \begin{align*}
        Z_t(u)= X_t(u_1,u_2)
    \end{align*}
    and we get 
    \begin{align*}
        \E\left((\sup_{0\le s \le t} \abs{Z_s(u)-Z_s(v)})^{2q}\right) \le C(d,q,K)t^q \abs{u-v}^{2q}
    \end{align*}
    Following the proof of Theorem 1.8.1 in \cite{kunita2019stochastic}, we get that there exists a version, which we will also denote by $Z_t(u)$ with $t\ge 0$ and $u\in [0,1]^{2d}$ such that 
    \begin{align*}
        \E\left((\sup_{u\in [0,1]^{2d}} \sup_{0\le t\le T} \abs{Z_t(u)})^{2q} \right)&=\E\left((\sup_{u\in [0,1]^{2d}} \sup_{0\le t\le T} \abs{Z_t(u)-Z_t(0)})^{2q} \right)\\
        &\le C \E\left(\left(\sumkinfty \omega^\beta_k(Z)\right)^{2q} \norm{u}^{\beta q}\right)\\
        &\le C(d,q,K)T^{2q}  
    \end{align*}
    for all $T\ge 0$ and some $0<\beta\le 1$. The last inequality follows from Lemma 1.8.2 in \cite{kunita2019stochastic} by observing  
    \begin{align*}
        \omega_k(Z) &= \max_{u,v\in \Delta_k: \norm{u-v}_\infty=2^{-k}} \sup_{0\le t\le T} \abs{Z_t(u)-Z_t(v)}\\
        \omega^\beta_k(Z) &= 2^{\beta k} \omega_k(Z)
    \end{align*}
     and $\Delta_k$ is defined as the set of all dyadic rationals of length $n$ on p.41 \cite{kunita2019stochastic}. Hence, by choosing $q$ large enough, one can show in exactly the same way as in Lemma \ref{MartglmKonvfs}
     \begin{align*}
         \lim_{t\to \infty } \frac{Y_t}{t} = 0 \fastsicher
     \end{align*}
     by the choice of $q$.
     Now choose $N_2$ as the set on which the convergence above fails. 
     
     From Lemma \ref{HölderBV} we get 
     \begin{align*}
         \frac{\abs{BV_t(u)-BV_t(v)}}{t}&\le \frac{\int_0^t \abs{x(u,s)-x(v,s)}^\delta \dd s}{t}\\
         &\le \left(\frac{\int_0^t \abs{x(u,s)-x(v,s)}^2 \dd s}{t}\right)^{\frac{\delta}{2}}  \\
         \overset{\eqref{Groenwall}}&{\le}\left(\frac{\int_0^t \exp(-(2\alpha-B^2)s)(\abs{u-v}^2+ Y_t) \dd s}{t}\right)^{\frac{\delta}{2}}
         \\&=  \left(\frac{(\abs{u-v}^2+ Y_t) (1-\exp(-(2\alpha-B^2t)))}{t}\right)^{\frac{\delta}{2}}\\
         &\le \left(\frac{(\abs{u-v}^2+ Y_t) }{t}\right)^{\frac{\delta}{2}} \fastsicher
     \end{align*}
  
    Hence for all $\omega \in \Omega\setminus (N_1\cup N_2)$ there exists a $T(\omega)>0$ such that for all $t\ge T(\omega)$ and $\eps >0$ there exists a $\gamma>0$ such that for all $u,v\in K$ with 
    \begin{align*}
        \abs{u-v}<\gamma \Rightarrow \frac{\abs{BV_{t}(u)(\omega) -BV_{t}(v)(\omega)}}{t} < \eps.
    \end{align*}
Now the Arzelá-Ascoli theorem yields the claim, since by Assumption \hyperlink{(A2)}{(A2)} $\left(\frac{BV_t}{t}\right)_{t\ge 1}$ is bounded. 
 \end{beweis}
Our initial goal was to ensure existence of Lyapunov exponents i.e. show that the limit of $f_t(u)$ exists, as $t\to\infty$.  Furthermore, since the Lyapunov exponents occur under an integral sign, we need this to be uniform in some sense. In the previous Lemmas we have already shown good compactness properties of $f_t(u)$, now we will identify the limit under certain conditions. 
\begin{lemma}\label{Lyapexpglmkonvfs}
    Assume \hyperlink{(A1)}{(A1)}, \hyperlink{(A2)}{(A2)} and \hyperlink{(A5)}{(A5)}. If furthermore for all $u\in \Rd$ 
    \begin{align*}
        \sumkinfty\sum_{p=1}^d \tr\left(\left(Db^{\cdot,p}_k(u,\delta_u)\right)^2\right) = L
    \end{align*}
    holds for some $L\in \R$, and the function
    \begin{align*}
          \sumkinfty\sum_{p=1}^d \tr\left(\left(Db^{\cdot,p}_k(\cdot,\cdot)\right)^2\right)
    \end{align*}
    is jointly continuous on $\Rd \times \M$. Then there exists $\lambda<0$ such that 
    \begin{align*}
        \lim_{t\to \infty} \sup_{u\in K} \abs{f_t(u)-\lambda} =0 \fastsicher
    \end{align*}
    for all compact $K\subset \Rd$. 
\end{lemma}
\begin{beweis}
    We will first show that for all $u,v\in \Rd$
    \begin{align*}
        \lim_{t\to \infty}\abs{x(u,t)-x(v,t)} = 0 \fastsicher
    \end{align*}
    holds. Keeping the notation of Lemma \ref{BVArzela} we 
 we obtain for all $u,v\in K$, due to \text{Lemma} \ref{Verschiebungskompakt} 
    \begin{align*}
        \lim_{t\to\infty}\abs{x(u,t)-x(v,t)}^2\le   \lim_{t\to\infty}\exp(-(2\alpha-B^2)t) (\abs{u-v}^2 + Y^K_t)  = 0\fastsicher.
    \end{align*}
    where $K\subset\Rd$ is compact. This follows from the proof of Lemma \ref{BVArzela}, where we have shown
    \begin{align*}
        \lim_{t\to \infty}\frac{Y_t}{t} = 0 \fastsicher. 
    \end{align*}
    Since we can exhaust $\Rd$ with countably many  compacts $(K_n)_{n\in \N}$ we get the result for all $u,v\in \Rd$. From this we get for all $v\in \Rd$ 
    \begin{align*}
        \lim_{t\to \infty}\gamma(\mu_t,\delta_{x(v,t)}) = 0 \fastsicher.
    \end{align*}
    Since
    \begin{align*}
      \lim_{t\to \infty} \gamma(\mu_t,\delta_{x(v,t)})\le\lim_{t\to \infty} \int_{\Rd} \frac{\abs{x(u,t)-x(v,t)}}{1+\abs{x(u,t)-x(v,t)}}\mu_0(\dd u) =0 \fastsicher.
    \end{align*}

Thus we get $u\in \Rd$
\begin{align*}
    \lim_{t\to \infty} \int_{\Rd} \dive(\phi(\cdot))(x(u,t)-x(v,t)) \mu_0(\dd u) = \dive(\phi(\cdot))(0) \fastsicher
\end{align*}
and 
\begin{align*}
    &\lim_{t\to \infty} \abs{\sumkinfty\sum_{p=1}^d \tr\left(\left(Db^{\cdot,p}_k(x(u,t),\mu_t)\right)^2\right)-L}\\
    &= \lim_{t\to \infty} \abs{\sumkinfty\sum_{p=1}^d \tr\left(\left(Db^{\cdot,p}_k(x(u,t),\mu_t)\right)^2\right)-\sumkinfty\sum_{p=1}^d \tr\left(\left(Db^{\cdot,p}_k(x(0,t),\delta_{x(0,t)})\right)^2\right)} = 0 \fastsicher
\end{align*}
Hence we can obtain by the rules of de l'Hôspital 
\begin{align*}
    \lim_{t\to \infty} \frac{BV_t(u)}{t} &= \lim_{t\to\infty} \int_{\Rd} \dive(\phi(\cdot))(x(u,t)-x(v,t)) \mu_0(\dd v) \\
    &- \frac{1}{2}\lim_{t\to\infty} \tr\left(\left(Db^{\cdot,p}_k(x(u,t),\mu_t)\right)^2\right)\\
    &= \dive(\phi(0)) -\frac{1}{2} L \fastsicher
\end{align*}
for the martingale part we obtained in Lemma \ref{MartglmKonvfs}
\begin{align*}
    \lim_{t\to \infty} \sup_{u\in K} \frac{M_t(u)}{t} = 0 \fastsicher.
\end{align*}
Hence, due to Lemma \ref{BVArzela}, by a simple subsequence argument we get 
\begin{align*}
    \lim_{t\to \infty} \sup_{u\in K} \abs{f_t(u)- \dive(\phi(0))+\frac{1}{2}L} = 0\fastsicher
 \end{align*}
 for all compact $K\subset\Rd $. The negativity of the limit, follows from Lemma \ref{Lyapneg}
\end{beweis}
\begin{bemerkung}
    The assumption for $\sum_{k=0}^\infty \sum_{p=1}^d\tr\left( \left(Db_k^{\cdot,p}(\cdot,\cdot)\right)^2 \right)$ is for example fulfilled when $b_k\equiv 0$ for all $k\ge 1$ and 
    \begin{align*}
        b_0(u,\mu) = \int_{\Rd}\dots\int_{\Rd}  \beta(u-v_1,\dots,u-v_n) \mu(\dd v_1),\dots, \mu(\dd v_n) 
    \end{align*}
    for $n\in \N$ and a differentiable function $\beta$ with bounded and $\delta$-Hölder derivative. 
\end{bemerkung}
\begin{bemerkung}
    The main difficulty here, is to guarantee the almost sure existence of Lyapunov exponents, pointwise. Whenever this achieved, one can conclude the result of Lemma \ref{Lyapexpglmkonvfs} holds.
\end{bemerkung}
With this Lemma at hand we can finally prove the main result.

\begin{theorem}
Under the assumptions \hyperlink{(A1)}{(A1)},\hyperlink{(A2)}{(A2)},\hyperlink{(A3)}{(A3)},\hyperlink{(A4)}{(A4)} and \hyperlink{(A5)}{(A5)} the random fields  $(p_t)_{t\ge 0}$ are intermittent if $p_0$ has compact support. 
\end{theorem}
\begin{beweis}
    We will denote $\overline{\supp(p_0)}=K$. First note that due to Lemma \ref{Lyapexpglmkonvfs}, fix the null set $N$ coming from Lemma \ref{Lyapexpglmkonvfs} and fix $\omega \in \Omega\setminus N $. Now there exists for all $\eps>0$ and all $u\in K$ a $T(\omega,\eps,K)>0$ such that for all $t\ge T(\omega,\eps,K)$  
    \begin{align*}
        &-\eps +\lambda\le f_t(u)(\omega) \le \eps +\lambda \\
        \Rightarrow& \exp\left(-\eps +\lambda t  \right)\le \dett\left(Dx(u,t)(\omega)\right)\le \exp\left(\eps +\lambda t  \right)
    \end{align*}
with $\lambda <0 $. Hence for all $t\ge T(\omega,\eps,K)$ we get 
\begin{align*}
    \frac{\ln(\int_{\Rd} p_t(u)^p(\omega) \dd u)}{t}&=  \frac{\ln(\int_{K} p_0(u)^p \frac{1}{\left(\dett(Dx(u,t))\right)^{p-1}(\omega)}\dd u)}{t} \\
    &\le  \frac{\ln(\exp((p-1)(\eps -\lambda t))\int_{K}p_0(u)^p \dd u}{t}) 
\end{align*}
and similarly 
\begin{align*}
\frac{\ln(\int_{\Rd} p_t(u)^p(\omega) \dd u)}{t}\ge  \frac{\ln(\exp((p-1)(-\eps -\lambda t))\int_{K}p_0(u)^p \dd u)}{t}.
\end{align*}
Since $\eps >0$ was arbitrary $\lambda_p$ exists and
\begin{align*}
    \lambda_p = -\lambda(p-1)
\end{align*}
since $\lambda<0$ we get that
\begin{align*}
    \left(\frac{\lambda_p}{p}\right)_{p\ge 1}
\end{align*}
is strictly increasing. 
\end{beweis}

\printbibliography
\end{document}